\documentclass[a4paper,11pt,reqno]{amsart}
\usepackage{a4,amsbsy,amscd,amssymb,amstext,amsmath,latexsym,oldgerm,enumerate,verbatim,graphicx,xy}

\newtheorem{thm} {Theorem} [section]
\newtheorem{prop}{Proposition} [section]
\newtheorem{lem} {Lemma} [section]
\newtheorem{lemgl} {Lemma}
\newtheorem{corgl} {Corollary}%in het geval van meer dan een corollarium gebruik ik het
\newtheorem{cornn}{Corollary}

\newtheorem{notnn}{Notation}

\theoremstyle{definition}

\newtheorem{rem} {Remark} [section]
\newtheorem{rems} [rem]{Remarks}

\newtheorem{remnn}{Remark}

\newtheorem{remsnn}{Remarks}

\newcommand{\mf}{\mathfrak}
\newcommand{\mc}{\mathcal}
\newcommand{\mb}{\mathbb}

\newcommand{\nts}{\negthinspace}     %handig
\newcommand{\Nts}{\nts\nts}

\newcommand{\ov}{\overline}

\newcommand{\sm}{\setminus}         %verzamelingen

\newcommand{\la}{\langle}
\newcommand{\ra}{\rangle}

\newcommand{\Hom}{{\rm Hom}}        %Algebra algemeen
\newcommand{\End}{{\rm End}}

\newcommand{\Mat}{{\rm Mat}}

\newcommand{\Ext}{{\rm Ext}}
\newcommand{\Tor}{{\rm Tor}}

\newcommand{\Sym}{{\rm Sym}} %symmetrische groep

\newcommand{\Ker}{{\rm Ker}}

\newcommand{\id}{{\rm id}}
\let\ttie\t
\newcommand{\tie}[1]{{\let\t\ttie \ttie#1}}%\t requires a special treatment, because
\renewcommand{\t}{\mf{t}}  %it is defined recursively. This trick is due to Uwe L\"uck

\newcommand{\GL}{{\rm GL}}

\newcommand{\Ort}{{\rm O}}
\newcommand{\SO}{{\rm SO}}

\newcommand{\Sp}{{\rm Sp}} %symplectische groep

\newcommand{\se}{\epsilon}
\newcommand{\e}{\varepsilon}
\newcommand{\schi}{\chi_{\hbox{${}_0$}}}
\newcommand{\rot}{\rotatebox}

\newcommand{\Max}{{\rm Max}}
\newcommand{\Grot}{{\rm Grot}}
\newcommand{\ch}{{\rm ch}}

\xyoption{all}

\begin{document}

\title{The Brauer algebra and the symplectic Schur algebra}

\begin{abstract}
Let $k$ be an algebraically closed field of characteristic $p>0$, let $m,r$ be integers with $m\ge1$, $r\ge0$ and $m\ge r$ and let $S_0(2m,r)$ be the symplectic Schur algebra over $k$ as introduced by the first author. We introduce the symplectic Schur functor, derive some basic properties of it and relate this to work of Hartmann and Paget. We do the same for the orthogonal Schur algebra. We give a modified Jantzen sum formula and a block result for the symplectic Schur algebra under the assumption that $r$ and the residue of $2m$ mod $p$ are small relative to $p$. From this we deduce a block result for the orthogonal Schur algebra under similar assumptions.
We also show that, in general, the block relations of the Brauer algebra and the symplectic and orthogonal Schur algebra are the same. Finally, we deduce from the previous results a new proof of the description of the blocks of the Brauer algebra in characteristic $0$ as obtained by Cox, De Visscher and Martin.
\end{abstract}

\author{Stephen Donkin {\tiny and} Rudolf Tange}
%\address{Department of Mathematics,
%University of York,
%Heslington, York, UK. YO10~5DD
%{\it E-mail addresses : }{\tt sd510@york.ac.uk, rht502@york.ac.uk}
%}

\maketitle

\section*{Introduction}
In characteristic zero the strong relationship between the representation theories of the general linear group and the symmetric group, is well-known; see e.g. \cite{Weyl}. In Green's monograph \cite{Green} a characteristic free approach to this is given, using the Schur functor as defined by Schur in his doctoral dissertation. In this work we give a systematic Lie theoretic approach to the representation theory of the Brauer algebra in the spirit of Green's monograph. Our approach was stimulated by a result of Cox, De Visscher and Martin \cite{CdVM2} expressing the blocks of the Brauer algebra in characteristic zero in terms of Weyl group orbits, and a desire to see this result in a Lie theoretic context. Throughout the paper we work over an algebraically closed field $k$.

The paper is organized as follows. In Section~\ref{s.prelim} we introduce the necessary notation, including the Brauer algebra $B_r=B_r(\delta)$ and for $n=2m$ even, the symplectic group $\Sp_n$ and the symplectic Schur algebra $S_0(n,r)$. Furthermore, we introduce Specht, permutation and Young modules for the Brauer algebra as in \cite{HarPag} and their twisted versions.

In Section~\ref{s.sympschur} we introduce the symplectic Schur functor $$f_0:{\rm mod}(S_0(n,r))\to{\rm mod}(B_r(-n))$$ and the inverse symplectic Schur functor $$g_0:{\rm mod}(B_r(-n))\to{\rm mod}(S_0(n,r)).$$ We assume that $m\ge r$ to ensure that the Brauer algebra $B_r(-n)$ can be identified with the endomorphism algebra $\End_{\Sp_n}(E^{\otimes r})$. Only in this situation can we expect the symplectic Schur functor to \lq\lq control" the representation theory of the Brauer algebra. %$Nonetheless, the study of the symplectic Schur functor in case $m<r$ remains an interesting open problem.
The main results are Theorem~\ref{thm.sympschur} and Propositions~\ref{prop.sympmult} and \ref{prop.sympyoung}. These results express the link between the representation theories of the symplectic Schur algebra and the Brauer algebra. They will be needed for the results in Section~\ref{s.blocks}.

In Section~\ref{s.jantzen} we study the symplectic Schur algebra in the situation that ${\rm char}\,k=p>2$ and $r$ and the residue of $n$ mod $p$ are small relative to $p$. This section is independent of Section~\ref{s.sympschur}. We obtain a description of the blocks, Theorem~\ref{thm.sympschurblocks}, which is the same as the description of the blocks of the Brauer algebra in characteristic zero in \cite{CdVM2}. Our main tool is a strengthened version of the Jantzen Sum Formula, Theorem~\ref{thm.sumformula}.

In Section~\ref{s.orthschur} we obtain the orthogonal versions of the results in Section~\ref{s.sympschur}. We assume here that ${\rm char}\,k\ne2$. Since we are assuming that $n>2r$, we can work with the special orthogonal group and avoid working with the full orthogonal group. Mostly the arguments are the same as in the symplectic case, and in that case they are omitted. The results in this section are important since they allow us to pass to the (untwisted) Specht, permutation and Young modules for the Brauer algebra via a Schur functor.

In Section~\ref{s.blocks}, we use the results of the previous three sections to obtain block results. First we assume that ${\rm char}\,k=p>2$ and show that the block relations of the symplectic and orthogonal Schur algebras are the same as those of the corresponding Brauer algebras, Theorem~\ref{thm.sameblocks}. From this we deduce generic block results for the Brauer algebra and the orthogonal Schur algebra. In Subsection~\ref{ss.CdVM} we assume that ${\rm char}\,k=0$ and deduce the description of the blocks of the Brauer algebra \cite[Thm.~4.2]{CdVM2}. For this we use reduction mod $p$. The key point about working in positive characteristic $p$ is that we can then subtract multiples of $p$ from $\delta$ to get that $\delta-up=-n=-2m$, $m\ge r$, without changing the Brauer algebra: $B_r(\delta)=B_r(-n)$. In the situation that $m\ge r$ we can then exploit the relation between the representation theories of the symplectic Schur algebra $S_0(n,r)$ and the Brauer algebra $B_r(-n)$. For the reduction mod $p$ to work we need that, for a fixed integer $\delta$, the blocks of the Brauer algebra over a field of characteristic zero \lq\lq agree" with the blocks over a field of large prime characteristic. This is a very general fact, as is explained in Subsection~\ref{ss.redmodp}. The idea that characteristic zero theory is the limiting case of characteristic $p$ theory has also been used for the partition algebra in \cite{MaWo}.

Throughout this paper we will freely make use of the general theory of quasihereditary algebras, the theory of reductive groups and their representations and the representation theory of the symmetric group. For quasihereditary algebras (e.g. (co)standard module, $\nabla$-filtration) we refer to \cite[Appendix]{Don7}. For reductive groups and their representations (e.g. induced module, Weyl module, good filtration) we refer to \cite[Part II]{Jan}. Note that in \cite{Jan} the induced and Weyl module with highest weight $\lambda$ for a reductive group are denoted by $H^0(\lambda)$ and $V(\lambda)$ respectively. For the representation theory of the symmetric group (e.g. Specht module, permutation module, $p$-regular partition) we refer to \cite{James}. A definition of Young modules for the symmetric group can be found in \cite{smartin}.

\section{Preliminaries}\label{s.prelim}
\subsection{The Brauer algebra and the symplectic Schur algebra}\label{ss.sympschur}
Throughout the paper $k$ denotes an algebraically closed field.
Let $n=2m$ be an even integer $\ge2$. Let $i\mapsto i'$ be the involution  of $\{1,\ldots,n\}$ defined by $i':=n+1-i$. Set $\se_i=1$ if $i\le m$ and $\se_i=-1$ if $i>m$ and define the $n\times n$-matrix $J$ with coefficients in $k$ by $J_{ij}=\delta_{ij'}\se_i$. So

%\begin{comment}
$$J=
\begin{bmatrix}
&&&&&1\\
&0&&&\rot{72}{$\ddots$}&\\
&&&1&&\\
&&-1&&&\\
&\rot{72}{$\ddots$}&&&0&\\
-1&&&&&
\end{bmatrix}.
$$
%\end{comment}
Let $E=k^n$ be the space of column vectors of length $n$ with standard basis $e_1,\ldots,e_n$.
On $E$ we define the nondegenerate symplectic form $\la\ ,\ \ra$ by
$$\la u,v\ra:=u^TJv=\sum_{i=1}^n\se_iu_iv_{i'}\ .$$
Then $\la e_i,e_j\ra=J_{ij}$. The {\it symplectic group} $\Sp_n=\Sp_n(k)$ is defined as the group of $n\times n$-matrices over $k$ that satisfy $A^TJA=J$, i.e. the invertible matrices for which the corresponding automorphism of $E$ preserves the form $\la\ ,\ \ra$. We denote the general linear group over $k$ by $\GL_n$ or $\GL_n(k)$. The vector space $E$ is the natural module for $\GL_n$ and for $\Sp_n$.

Let $r$ be an integer $\ge0$. For any $\delta\in k$ one has the Brauer algebra $B_r(\delta)$; see e.g. \cite{Br}, \cite{Brown1}, \cite{DorHanWal}, \cite{Wen} or \cite{CdVM1} for a definition. This also makes sense for $\delta$ an integer, since we can replace that integer by its natural image in $k$. Let $E^{\otimes r}$ be the $r$-fold tensor power of $E$. Then we have natural homomorphisms $k\Sym_r\to\End_{\GL_n}(E^{\otimes r})$ and $B_r(-n)\to\End_{\Sp_n}(E^{\otimes r})$. The action of the symmetric group $\Sym_r$ is by permutation of the factors, the action of $B_r(-n)$ is explained in \cite[p 192]{Wen}. Using classical invariant theory one can then show that these homomorphisms are surjective and that they are injective in case $n\ge r$ and $m\ge r$, respectively; see \cite{DeCProc} and \cite{T} (the proof of the surjectivity is Brauer's original argument \cite{Br}). In the case of the symplectic group these results were first obtained in arbitrary characteristic in \cite{Dotyetal} using different methods. Let $S(n,r)$ and $S_0(n,r)$ be the enveloping algebras in $\End(E^{\otimes r})$ of $\GL_n$ and $\Sp_n$ respectively. Then the natural embeddings $S(n,r)\to\End_{k\Sym_r}(E^{\otimes r})$ and $S_0(n,r)\to\End_{B_r(-n)}(E^{\otimes r})$ are isomorphisms; see \cite{Green} and \cite{T}. In the case of the symplectic group this result was first obtained in arbitrary characteristic by Oehms \cite{Oe}. The proof given in \cite{T} avoids the FRT-construction, but makes essential use of the bideterminant basis given in \cite[Thm.~6.1]{Oe}. The algebra $S(n,r)$ is the {\it Schur algebra} as introduced in \cite{Green} and we will call $S_0(n,r)$ the {\it symplectic Schur algebra}, it was first introduced in \cite{Don3}.

In \cite{Don1} {\it generalized Schur algebras} were introduced. These are quasi-hereditary finite dimensional algebras that are associated to a reductive group and a finite saturated set of weights. Let $T$ and $T_0$ be the maximal tori of $\GL_n$ and $\Sp_n$, respectively, that consist of diagonal matrices and let $B$ and $B_0$ be the Borel subgroups that consist of upper triangular matrices. Note that $T_0$ consists of those $t\in T$ which satisfy $t_it_{i'}=1$ for all $i\in\{1,\ldots,m\}$. Associated to a maximal torus and a Borel subgroup containing it one has a root datum and a choice of positive roots. We call the characters (multiplicative one-parameter subgroups) of a fixed maximal torus of a reductive group {\it weights}. Recall that a weight $\lambda$ is called {\it dominant} if $\la\lambda,\alpha^\vee\ra\ge0$ for every positive root $\alpha$.
We denote the set of weights of $\GL_n$ with respect to $T$ by $X$ and the set of weights of $\Sp_n$ with respect to $T_0$ by $X_0$. The groups $X$ and $X_0$ can be identified with $\mb Z^n$ and $\mb Z^m$, respectively.

\medskip
For $l<n$ we identify $\mb Z^l$ with the sublattice of $\mb Z^n$ that consists of the $n$-tuples with the last $n-l$ components equal to $0$.
\medskip

For $i\in\{1,\ldots,n\}$ we denote the element of $\mb Z^n$ which is $1$ on the $i^{\rm th}$ position and $0$ elsewhere by $\e_i$. The character corresponding to $\e_i$ is for $\GL_n$, and also for $\Sp_n$ if $i\le m$, the $i^{\rm th}$ diagonal entry function. The map $\lambda\mapsto\ov\lambda: \mb Z^n\to\mb Z^m$ corresponding to restriction of characters sends $\e_i$ to $\e_i$ if $i\le m$ and to $-\e_{i'}$ if $i>m$.

Recall that a function on a closed subgroup of $\GL_n$ is called {\it polynomial} if it is the restriction of a function on $\GL_n$ which is a polynomial in the matrix entries. Clearly all regular functions on $\Sp_n$ and $T_0$ are polynomial. The set of polynomial weights of $\GL_n$ with respect to $T$ corresponds, under the above identifications, to the subset $\mb N^n$ of $\mb Z^n$ that consists of the compositions $\lambda_1,\lambda_2,\cdots,\lambda_n\ge0$ of some integer $|\lambda|:=\sum_i\lambda_i$ into at most $n$ parts. The set of dominant polynomial weights with respect to $T$ and $B$ corresponds to the subset $\Lambda^+(n)$ of $\mb Z^n$ that consists of the partitions $\lambda_1\ge\lambda_2\ge\cdots\ge\lambda_n\ge0$ of some integer into at most $n$ parts.
The set of dominant weights of $\Sp_n$ with respect to $T_0$ and $B_0$ corresponds to the set $\Lambda_0^+(m)=\Lambda^+(m)\subseteq Z^m$. Define the sets of weights $\Lambda(n,r)$, $\Lambda^+(n,r)$, $\Lambda_0(m,r)$ and $\Lambda_0^+(m,r)$ by
\begin{align*}
\Lambda(n,r):=&\{\lambda\in \mb N^n\,|\,|\lambda|=r\},\ \Lambda^+(n,r)=\Lambda(n,r)\cap\Lambda^+(n)\\
\Lambda_0(m,r):=&\{\lambda\in \mb Z^m\,|\,-r\le|\lambda|\le r, r-|\lambda|\text{ even}\}\text{\ and\ }\\
\Lambda^+_0(m,r):=&\Lambda_0(m,r)\cap\Lambda^+(m)=\{\lambda\in\Lambda^+(m)\,|\,|\lambda|\le r, r-|\lambda|\text{ even}\}.
\end{align*}
Note that, under the above identifications, $\Lambda^+(n,r)=\Lambda^+(r,r)$ if $n\ge r$ and that $\Lambda^+_0(m,r)=\Lambda^+_0(r,r)$ if $m\ge r$.
The algebra $S(n,r)$ is the generalized Schur algebra associated to $\GL_n$ and $\Lambda^+(n,r)$ by \cite[Thm 8.3]{Don8}. Furthermore, one can deduce in the same way that $S_0(n,r)$ is the generalized Schur algebra associated to $\Sp_n$ and $\Lambda_0^+(m,r)$. Here we prefer to work with the symplectic group rather than the symplectic similitude group as in \cite{Don3}. For $S(n,r)$ we denote the standard, costandard and irreducible module associated to $\lambda\in\Lambda^+(n,r)$ by $\Delta(\lambda)$, $\nabla(\lambda)$ and $L(\lambda)$, respectively. Occasionally, we will also use this notation for an arbitrary quasi-hereditary algebra or an arbitrary connected reductive group. For $S_0(n,r)$ and $\lambda\in\Lambda_0^+(m,r)$ we denote these modules by $\Delta_0(\lambda)$, $\nabla_0(\lambda)$ and $L_0(\lambda)$. In case of $S(n,r)$, $S_0(n,r)$ or a connected reductive group these are the Weyl, induced and irreducible module associated to $\lambda$ for the group. Here we induce from the opposite Borel subgroups of $B$ and $B_0$ to $\GL_n$ and $\Sp_n$.

Later on we will need the following lemma which is, no doubt, well-known. Except for the second assertion of (i) (this can be found in \cite[Thm.~2.1]{BenCa} and \cite{Serre2}, for example), we couldn't find the result in the literature, so we include a proof.

\begin{lem}\label{lem.directsummand}\
\begin{enumerate}[{\rm(i)}]
\item Let $M$ be a finite dimensional vector space over $k$. The $k\GL(M)$-module $M$ is a direct summand of $M\otimes M^*\otimes M$ and if $\dim M\ne0$ in $k$, then the trivial $k\GL(M)$-module $k$ is a direct summand of $M\otimes M^*$.
\item Let $G$ be a group and let $M$ be a self-dual finite dimensional $kG$-module. Let $r$ and $t$ be integers with $0\le t\le r$ and $r-t$ even. Then $M^{\otimes t}$ is a direct summand of $M^{\otimes r}$ if $t\ge1$ or $\dim M\ne0$ in $k$.
\end{enumerate}
\end{lem}
\begin{proof}
(i).\ Put $l=\dim M$, let $(v_1,\ldots,v_l)$ be a basis of $M$ and let $(v_1^*,\ldots,v_l^*)$ be the dual basis of $M^*$. Let $a:k\to M\otimes M^*$ be the map $\alpha\mapsto \alpha\sum_{i=1}^lv_i\otimes v_i^*$, let $b:M\to M\otimes M^*\otimes M$ be the map $x\mapsto x\otimes\sum_{i=1}^lv_i^*\otimes v_i$ and let $c:M\otimes M^*\to k$ be the contraction by means of the canonical bilinear form. Then one easily checks that $(c\otimes\id)\circ b=\id$ and that $c\circ a=l\,\id$. This proves (i).\\
(ii).\ By (i) we have that $M$ is a direct summand of $M^{\otimes 3}$ and, if $\dim M\ne0$ in $k$, the trivial $kG$-module $k$ is a direct summand of $M^{\otimes 2}$. The assertion now follows by induction.\\
\end{proof}

\subsection{Modules for the Brauer algebra}\label{ss.braueralgebra}
\begin{notnn}
In what follows, $s$ and $t$ are not necessarily fixed integers $\ge0$ such that $r=t+2s$.
\end{notnn}
Let $\delta\in k$. For any integer $i\ge0$, let $I_{s,i}$ be the left ideal of the Brauer algebra $B_r=B_r(\delta)$ spanned by the diagrams of which the bottom row has $s$ horizontal edges which each join two consecutive nodes of the last $2s$ nodes and has at least $i$ other horizontal edges. Put $I_s:=I_{s,0}$, $Z_{s,i}:=I_{s,i}/I_{s,i+1}$ and $Z_s=Z_{s,0}$. Note that $I_{s,i}=Z_{s,i}=0$ if $s+i>r/2$. The symmetric group $\Sym_t$ acts on $I_s$ from the right by permuting the first $t$ nodes of the bottom row of a diagram. Thus $I_s$ and $Z_s$ are $(B_r(\delta), k\Sym_t)$-bimodules. Furthermore $Z_s$ is a free right $k\Sym_t$-module which has as a basis the canonical images of the diagrams in which the vertical edges do not cross and of which the bottom row has precisely $s$ horizontal edges which each join two consecutive nodes of the last $2s$ nodes. One easily checks that there are $$\frac{r!}{s!t!2^s}$$ such diagrams.

Let $\lambda$ be a partition of $t$ and let $S(\lambda)$, $M(\lambda)$ and $Y(\lambda)$ be the Specht module, permutation module and Young module of $k\Sym_t$ associated to $\lambda$. If ${\rm char}\,k=0$, then $S(\lambda)$ is irreducible and we also denote it by $D(\lambda)$. If ${\rm char}\,k=p>0$ and $\lambda$ is $p$-regular, then $S(\lambda)$ has a simple head and we denote it by $D(\lambda)$. Denote the sign representation of $k\Sym_t$ by $k_{\rm sg}$.

Following \cite{DorHanWal} (see also \cite{HarPag}), we define the {\it Specht (or cell) module} $\mc S(\lambda)$ and {\it twisted Specht (or cell) module} $\widetilde{\mc S}(\lambda)$ for the Brauer algebra by
\begin{align*}
\mc S(\lambda):=&\,Z_s\otimes_{k\Sym_t}S(\lambda)\text{\quad and}\\
\widetilde{\mc S}(\lambda):=&\,Z_s\otimes_{k\Sym_t}\big(k_{\rm sg}\otimes S(\lambda)\big).
\end{align*}
By the above, $\dim\mc S(\lambda)=\dim\widetilde{\mc S}(\lambda)=\frac{r!}{s!t!2^s}\dim S(\lambda)$. By \cite[Rem.~6.4]{Green} we have $k_{\rm sg}\otimes S(\lambda)\cong S(\lambda')^*$, where $\lambda'$ denotes the transpose of $\lambda$. If ${\rm char}\,k=0$ or $>t$, then $S(\lambda)^*\cong S(\lambda)$ and $\widetilde{\mc S}(\lambda)\cong{\mc S}(\lambda')$.
\begin{remnn}
The definitions and results in \cite{HarPag} have obvious \lq\lq twisted versions" and in what follows we will also cite \cite{HarPag} for those twisted versions.
\end{remnn}
Following Hartmann and Paget \cite{HarPag}, we define the {\it permutation module} $\mc M(\lambda)$ and the {\it twisted permutation module} $\widetilde{\mc M}(\lambda)$ for the Brauer algebra by
\begin{align*}
\mc M(\lambda):=&\,{\rm Ind}^{B_r}_{k\Sym_t}M(\lambda)\text{\quad and}\\
\widetilde{\mc M}(\lambda):=&\,{\rm Ind}^{B_r}_{k\Sym_t}\big(k_{\rm sg}\otimes M(\lambda)\big).
\end{align*}
Here ${\rm Ind}^{B_r}_{k\Sym_t}$ is defined by ${\rm Ind}^{B_r}_{k\Sym_t}V=I_s\otimes_{k\Sym_t}V$ for any $k\Sym_t$-module $V$.
Note that $\widetilde{\mc M}((1^r))\cong B_r$, since $k_{\rm sg}\otimes k\Sym_r\cong k\Sym_r$ as $k\Sym_r$-modules. If $\lambda$ is $p$-regular and $\lambda\ne\emptyset$ in case $r$ is even $\ge2$ and $\delta=0$, then ${\mc S}(\lambda)$ and $\widetilde{\mc S}(\lambda)$ have a simple head which we denote by ${\mc D}(\lambda)$ and $\widetilde{\mc D}(\lambda)$. Whenever we write ${\mc D}(\lambda)$ and $\widetilde{\mc D}(\lambda)$ for some $p$-regular $\lambda$, we assume that $\lambda\ne\emptyset$ in case $r$ is even $\ge2$ and $\delta=0$.

In \cite{CdVM2}, after Lemma~2.1 and in Section~8, it is pointed out that it can be shown by completely elementary arguments that the above partitions form a labeling set for the irreducible $B_r$-modules. The exception $\lambda\ne\emptyset$ in case $r$ is even $\ge2$ and $\delta=0$ is caused by the fact that in case $\delta=0$, $B_2$ has a one-dimensional nilpotent ideal with quotient isomorphic to $k\Sym_2$.

Finally, we define the {\it Young module} $\mc Y(\lambda)$ and the {\it twisted Young} module $\widetilde{\mc Y}(\lambda)$ for the Brauer algebra as the unique indecomposable summand of $\mc M(\lambda)$ (resp. $\widetilde{\mc M}(\lambda)$) which surjects onto $Z_s\otimes_{k\Sym_t}Y(\lambda)$ (resp. $Z_s\otimes_{k\Sym_t}\big(k_{\rm sg}\otimes Y(\lambda)\big)$); see \cite[Def.~6.3]{HarPag}. There it is also observed that $\mc Y(\lambda)$ and $\widetilde{\mc Y}(\lambda)$ are actually indecomposable summands of ${\rm Ind}^{B_r}_{k\Sym_t}Y(\lambda)$ and ${\rm Ind}^{B_r}_{k\Sym_t}\big(k_{\rm sg}\otimes Y(\lambda)\big)$.

Let $i$ be an integer $\ge0$. The stabilizer of $\{\{1,2\},\ldots,\{2i-1,2i\}\}$ in $\Sym_{2i}$ is isomorphic to the hyperoctahedral group of degree $i$ and order $2^ii!$ and we denote it by $H_i$. We consider $\Sym_{2i}$ and $H_i$ as embedded in $\Sym_t$ via the embedding $\Sym_{t-2i}\times\Sym_{2i}\subseteq\Sym_t$. From the proof of \cite[Prop.~7.3]{HarPag} we deduce the following

\begin{prop}[{cf. proof of \cite[Prop.~7.3]{HarPag}}]\label{prop.filtration}
Let $V$ be a $k\Sym_t$-module.
\begin{enumerate}[{\rm(i)}]
\item $W:={\rm Ind}^{B_r}_{k\Sym_t}V$ has a descending filtration $W=W_0\supseteq W_1\supseteq\cdots$ such that $W_i=0$ for $i>{\lfloor t/2\rfloor}$ and $W_i/W_{i+1}\cong Z_{s,i}\otimes_{k\Sym_t}V$ for $i\ge0$.
\item $Z_{s,i}\otimes_{k\Sym_t}V\cong Z_{s+i}\otimes_{k\Sym_{t-2i}}V_{H_i}$ for $i\le{\lfloor t/2\rfloor}$, where $V_{H_i}$ is the largest trivial $H_i$-module quotient of $V$.
\end{enumerate}
\end{prop}
The filtration of ${\rm Ind}^{B_r}_{k\Sym_t}V=I_s\otimes_{k\Sym_t}V$ is constructed as follows. Let $I_s(i)$ be the subspace of $I_s$ spanned by the diagrams of which the bottom row has $s$ horizontal edges which each join two consecutive nodes of the last $2s$ nodes and has precisely $i$ other horizontal edges. Then $I_{s,i}=\bigoplus_{j\ge i}I_s(j)$. Since each $I_s(i)$ is stable under the right action of $\Sym_t$ on $I_s$, we have ${\rm Ind}^{B_r}_{k\Sym_t}V=\bigoplus_{i\ge0}(I_s(i)\otimes_{k\Sym_t}V)$. Now we put $W_i=\bigoplus_{j\ge i}(I_s(j)\otimes_{k\Sym_t}V)\cong I_{s,i}\otimes_{k\Sym_t}V$ and observe that $W_i$ is a $B_r$-submodule of $W$.

We record the following consequence of \cite[Prop.~6.1]{CdVM2} which was mentioned to us by A.~Cox. It shows that we can restrict to the case that $\delta$ lies in the prime field. Of course, a sharper result is known in characteristic $0$; see \cite{Wen} and \cite{Brown2}.
\begin{prop}[{cf.~\cite[Prop.~6.1]{CdVM2}}]\label{prop.notinFp}
Assume that $\delta$ does not lie in the prime field. Put $N_i=r!/\big(i!(r-2i)!2^i\big)$. Then
$$B_r(\delta)\cong\bigoplus_{i=0}^{\lfloor r/2\rfloor}\Mat_{N_i}(k\Sym_{r-2i}).$$
\end{prop}
\begin{proof}
Let $i\in\{0,\ldots,\lfloor r/2\rfloor\}$ and put $t_i=r-2i$. Let $J_i$ be the two-sided ideal of $B_r$ that is spanned by the diagrams which have at least $2i$ horizontal edges. Note that $J_i=I_iB_r$. By \cite[Prop.~6.1]{CdVM2} we have that $p$-regular partitions of different numbers belong to different blocks. Since cell modules always belong to one block, we get that ${\mc S}(\lambda)$ can only have composition factors ${\mc D}(\mu)$, $\mu$ a $p$-regular partition of $|\lambda|$. From Proposition~\ref{prop.filtration} we now deduce that $I_i={\rm Ind}^{B_r}_{k\Sym_{t_i}}k\Sym_{t_i}$ has only composition factors ${\mc D}(\mu)$, $\mu$ a $p$-regular partition with $|\mu|\le t_i$. The same must hold for $J_i$, since it is a sum of images of $I_i$. By the proof of \cite[Prop.~3.3]{HarPag} the irreducible module ${\mc D}(\mu)$, $\mu$ $p$-regular, is killed by $I_i$ if and only if $|\mu|>t_i$. Since the composition factors of $B_r/J_i$ are all killed by $J_i$ they must be of the form ${\mc D}(\mu)$, $\mu$ $p$-regular with $|\mu|>t_i$. By \cite[Prop.~6.1]{CdVM2} there exists a left ideal $J^i$ of $B_r$ such that $B_r=J_i\oplus J^i$. Since we are dealing with the left regular module it is clear that $J^i$ must be a two-sided ideal. It follows that $B_r(\delta)\cong\bigoplus_{i=0}^{\lfloor r/2\rfloor}J_i/J_{i+1}$, where each of the algebras $J_i/J_{i+1}$ has a unit element. The algebra $J_i/J_{i+1}$ is isomorphic to $\Mat_{N_i}(k\Sym_{t_i})$ where the multiplication is given by $A\circ B=AXB$ for some fixed $X\in\Mat_{N_i}(k\Sym_{t_i})$; see \cite{Brown1} and \cite[Sect.~4]{koxi}. Since $J_i/J_{i+1}$ has a unit element we must have that $X$ is invertible in $\Mat_{N_i}(k\Sym_{t_i})$. But then $A\mapsto AX$ defines an isomorphism $J_i/J_{i+1}\stackrel{\sim}{\to}\Mat_{N_i}(k\Sym_{t_i})$.
\end{proof}

In the remainder of this subsection we assume that $\delta=-n=-2m$ and that $m\ge r$. Note that if the field $k$ is of prime characteristic $p>2$, then any element of the prime subfield of $k$ can be represented by an integer of this form. Under the representation $B_r\to\End_{\Sp_n}(E^{\otimes r})$ each Brauer diagram corresponds to an endomorphism of the $\Sp_n$-module $E^{\otimes r}$. There is a more direct way to associate to each Brauer diagram an endomorphism of the $\Sp_n$-module $E^{\otimes r}$; see e.g. \cite[p 871]{Br} or \cite[Sect.~3]{T}. Furthermore, there is a unique algebra structure on the vector space $B_r$ such that this other map is a homomorphism of algebras. Let us call the resulting algebra the symplectic Brauer algebra and denote it by $\widetilde{B}_r(n)$ or just $\widetilde{B}_r$. Of course, the other map comes from an isomorphism $\widetilde{B}_r(n)\stackrel{\sim}{\to}B_r(-n)$. This isomorphism sends each of the $r$ standard generators of $\widetilde{B}_r$ to the negative of the corresponding standard generator of $B_r$. This implies that each diagram $d\in\widetilde{B}_r(n)$ corresponds to $\pm d\in B_r(-n)$. The multiplication of $\widetilde{B}_r$ is more complicated to describe. In \cite{Br} Brauer introduced the algebras $B_r(n)$ and $\widetilde{B}_r(n)$ and their action on tensor space separately, the isomorphism $\widetilde{B}_r(n)\stackrel{\sim}{\to}B_r(-n)$ was observed later in \cite{HanWal}. In \cite[Sect.~3]{T} it is explained, using Brauer's original arguments, that, more generally, $\Hom_{\Sp_n}(E^{\otimes t_2},E^{\otimes t_1})$ has a basis indexed by $(t_1,t_2)$-diagrams. These are diagrams which are graphs whose vertices are arranged in two rows, $t_1$ in the top row and $t_2$ in the bottom row, and whose edges form a matching of the $t_1+t_2$ nodes in (unordered) pairs. The horizontal edges in the bottom row correspond to contractions by means of the symplectic form and the horizontal edges in the top row correspond to \lq\lq multiplications" by the symplectic invariant $\sum_{i=1}^n\se_ie_i\otimes e_{i'}$. In the proofs of Lemmas~\ref{lem.surjective} and \ref{lem.tensor} below we will use these diagram bases.

The symplectic form on $E$ induces a nondegenerate bilinear form on $E^{\otimes r}$. So $\End_k(E^{\otimes r})$ has a transpose map. Recall that $B_r$ has a standard anti-automorphism $\iota$ that flips a diagram over the horizontal axis. One easily checks that $\iota(b)$ acts as the transpose of $b$ for all $b\in B_r$. This means that the $B_r$-module $E^{\otimes r}$ is self-dual. Under the isomorphism $B_r\stackrel{\sim}{\to}\widetilde{B}_r$ the left ideal $I_s$ is mapped to a left ideal $\widetilde{I}_s$ of $\widetilde{B}_r$ which is spanned by the same diagrams. These diagrams are in 1-1 correspondence with the $(r,t)$-diagrams: just omit the last $2s$ nodes in the bottom row and the edges which have these nodes as endpoints. So the canonical isomorphism $\widetilde{B}_r\stackrel{\sim}{\to}\Hom_{\Sp_n}(E^{\otimes r})$ induces a canonical isomorphism
$$I_s\stackrel{\sim}{\to}\Hom_{\Sp_n}(E^{\otimes t},E^{\otimes r})\otimes k_{\rm sg}$$ of $(B_r,k\Sym_t)$-bimodules. The vector space $\Hom_{\Sp_n}(E^{\otimes r},E^{\otimes t})$ has a natural $(k\Sym_t,B_r)$-bimodule structure and therefore, by means of the standard anti-automorphisms of $\Sym_t$ and $B_r$, also a natural $(B_r,k\Sym_t)$-bimodule structure. Composing the above isomorphism with the transpose map $\Hom_{\Sp_n}(E^{\otimes t},E^{\otimes r})\to\Hom_{\Sp_n}(E^{\otimes r},E^{\otimes t})$ we obtain an canonical isomorphism
\begin{equation}\label{eq.I_siso}
\varphi:I_s\stackrel{\sim}{\to}\Hom_{\Sp_n}(E^{\otimes r},E^{\otimes t})\otimes k_{\rm sg}
\end{equation}
of $(B_r,k\Sym_t)$-bimodules, which induces an isomorphism
\begin{equation}\label{eq.Z_siso}
Z_s\stackrel{\sim}{\to}\big(\Hom_{\Sp_n}(E^{\otimes r},E^{\otimes t})/\varphi(I_{s,1})\big)\otimes k_{\rm sg}
\end{equation}
of $(B_r,k\Sym_t)$-bimodules.

\section{The symplectic Schur functor}\label{s.sympschur}

For a finite dimensional algebra $A$ over $k$, we denote the category of finite dimensional $A$-modules by ${\rm mod}(A)$. The category ${\rm mod}(S(n,r))$ can be identified with the category of $\GL_n$-modules whose coefficients are homogeneous polynomials of degree $r$ in the matrix entries. Assume that $n\ge r\ge0$. The Schur functor $f:{\rm mod}(S(n,r))\to{\rm mod}(k\Sym_r)$ can be defined by
\begin{align*}
&f(M)=\Hom_{S(n,r)}(E^{\otimes r},M)=\Hom_{\GL_n}(E^{\otimes r},M).
\end{align*}
Here the action of the symmetric group comes from the action on $E^{\otimes r}$ and we use the inversion to turn right modules into left modules. An equivalent definition is: $f(M)=M_{\varpi_r}$, the weight space corresponding to the weight $\varpi_r=(1^r)=(1,1,\ldots,1)\in\mb Z^r\subseteq\mb Z^n$; see \cite{Green}. The isomorphism
\begin{align}\label{eq.schuriso}
\Hom_{\GL_n}(E^{\otimes r},M)\stackrel{\sim}{\to} M_{\varpi_r}
\end{align}
is given by $u\mapsto u(e_1\otimes e_2\otimes\cdots\otimes e_r)$. This can be deduced from \cite[6.2g Rem.~1 and 6.4f ]{Green}. We have embeddings $\Sym_r\subseteq\Sym_n\subseteq N_{\GL_n}(T)$, where the second embedding is by permutation matrices. Then $\varpi_r$ is fixed by $\Sym_r$, so there is an action of $\Sym_r$ on $M_{\varpi_r}$ for every $S(n,r)$-module $M$. With this action \eqref{eq.schuriso} is $\Sym_r$-equivariant. The inverse Schur functor $g:{\rm mod}(k\Sym_r)\to{\rm mod}(S(n,r))$ can be defined by
$$g(V)=E^{\otimes r}\otimes_{k\Sym_r}V.$$

We now retain the notation and assumptions of Subsection~\ref{ss.braueralgebra}. So $n=2m$, $m\ge r$ and $B_r=B_r(-n)$. We define the {\it symplectic Schur functor}
$$f_0:{\rm mod}(S_0(n,r))\to{\rm mod}(B_r)$$ by
\begin{align*}
&f_0(M)=\Hom_{S_0(n,r)}(E^{\otimes r},M)=\Hom_{\Sp_n}(E^{\otimes r},M).
\end{align*}
Here the action of the Brauer algebra comes from the action on $E^{\otimes r}$ and we use the standard anti-automorphism of $B_r$ to turn right modules into left modules. Note that the action of $\Sym_r$ on $E^{\otimes r}$ inherited from that of $B_r$ is its natural action twisted by the sign. Since $E=\nabla_0(\e_1)=\Delta_0(\e_1)$ is a tilting module, the same holds for $E^{\otimes r}$; see e.g. \cite[Prop.~1.2]{Don6}. This implies that $f_0$ maps exact sequences of modules with a good filtration to exact sequences.

We define the {\it inverse symplectic Schur functor}
$$g_0:{\rm mod}(B_r)\to{\rm mod}(S_0(n,r))$$ by
\begin{align*}
&g_0(V)=E^{\otimes r}\otimes_{B_r}V.
\end{align*}

By \cite[Thm~2.11]{Rot} we have for $V\in{\rm mod}(B_r)$ and $M\in{\rm mod}(S_0(n,r))$
\begin{equation}\label{eq.adjointiso}
\Hom_{\Sp_n}(g_0(V),M)\cong\Hom_{B_r}(V,f_0(M)).
\end{equation}
There is an alternative for $f_0$ and $g_0$:
$$\tilde f_0(M)=E^{\otimes r}\otimes_{S_0(n,r)}M\text{\quad and\quad }\tilde{g_0}(V)=\Hom_{B_r}(E^{\otimes r},M).$$
But, by \cite[Lemma~3.60]{Rot}, we have $\tilde{f}_0(V^*)\cong f_0(V)^*$ and $\tilde{g}_0(V^*)\cong g_0(V)^*$.
%The general iso is ($R$ a $k$-algebra) $M^*\otimes_RN\stackrel{\sim}{\to}\Hom_R(N,M)^*$; it sends $\lambda\otimes x$ to the functional that sends $u$ to $\lambda(u(x))$.
So the results obtained using $\tilde{f}_0$ and $\tilde{g}_0$ can also be obtained by dualizing the results obtained using $f_0$ and $g_0$. We sketch a proof of the following exercise in Brauer's Formula.

\begin{lem}\label{lem.dimension}
Let $\lambda$ be a partition of $t=r-2s$. Then
$$\dim\Hom_{\Sp_n}(\Delta_0(\lambda),E^{\otimes r})=\dim\Hom_{\Sp_n}(E^{\otimes r},\nabla_0(\lambda))=\frac{r!}{s!t!2^s}\dim S(\lambda).$$
\end{lem}
\begin{proof}
First note that, since $E^{\otimes r}$ has a good filtration, as an $\Sp_n$-module, the dimension of $\Hom_{\Sp_n}(\Delta_0(\lambda),E^{\otimes})$ is equal to the multiplicity of $\nabla_0(\lambda)$ in a good filtration of $E^{\otimes r}$. Similar remarks apply to the dimension of $\Hom_{\Sp_n}(E^{\otimes r},\nabla_0(\lambda))$. For a partition $\lambda$, with at most $m$ parts, we write $\schi(\lambda)$ for the formal character of $\nabla_0(\lambda)$. We have to show that the coefficient of $\schi(\lambda)$ in an expression of $\ch\,E^{\otimes r}$ as a ${\mb Z}$-linear combination of Weyl characters, is
$\displaystyle{\frac{r!}{s!t!2^s}}\dim S(\lambda)$. For a partition
$\lambda$ of $r$ we denote by $\chi^\lambda$ the corresponding character of $\Sym_r$. Then, for $r\geq 1$,  by the branching rule, we have
$$\chi^\lambda\downarrow_{\Sym_r}^{\Sym_{r+1}}=\sum_\mu \chi^\mu$$
where the sum is over all partitions $\mu$ of $r$ such that the diagram of $\mu$ is obtained from that of $\lambda$ by the removal of one box. By Frobenius reciprocity we have
$$\chi^\lambda\uparrow_{\Sym_{r-1}}^{\Sym_r}=\sum_\mu \chi^\mu$$
where $\mu$ ranges over those partitions of $r$ whose diagram is obtained by adding one box. In particular the degree of $\chi^\lambda$ is given by
\begin{align*}
r\deg(\chi^\lambda)&=\deg(\chi^\lambda\uparrow_{\Sym_{r-1}}^{\Sym_r})\\
&=\sum_\mu \deg(\chi^\mu).
\end{align*}

Now we define $\psi_0=1$ and for  $1\leq r\leq m$ we define  $$\psi_r=\sum_\lambda \deg(\chi^\lambda)\schi(\lambda)$$  where the sum is over all partitions of $r$. We claim that, for $1\leq r<m$, we have
\begin{equation}
\schi(1)\psi_r=\psi_{r+1}+r\psi_{r-1}.\tag{*}
\end{equation}
By Brauer's Formula (see e.g. \cite[Lem. II.5.8 b)]{Jan}) we have  $$\schi(1)\psi_r=\sum_{i=1}^m\sum_\mu\deg(\chi^\mu)\schi(\mu+\se_i)+\sum_{i=1}^m\sum_\mu
\deg(\chi^\mu)\schi(\mu-\se_i),$$
where in both sums $\mu$ ranges over partitions of $r$. One easily checks that $\schi(\mu\pm\se_i)\ne0$ if and only if $\mu\pm\se_i$ is a partition. %cf. the criterion for $\chi_\lambda)$ to be nonzero in the  proof of Theorem~\ref{thm.sumformula}
For a partition $\lambda$ of $r+1$, we see that the coefficient of $\schi(\lambda)$ in $\schi(1)\psi_r$ is $\sum_\mu \deg(\chi^\mu)$, where $\mu$ ranges over partitions of $r$ such that the diagram of $\mu$ is obtained from the diagram of $\lambda$ by removing a box. So this coefficient is $\deg(\chi^\lambda)$. For a partition $\lambda$ of $r-1$, the coefficient of $\schi(\lambda)$ in $\schi(1)\psi_r$ is $\sum_\mu\deg(\chi^\mu)$, where $\mu$ ranges over all partitions of $r$ such that the diagram of $\mu$ is obtained from the diagram of $\lambda$ by adding a box. So the coefficient of $\schi(\lambda)$ is $r\deg(\chi^\lambda)$. This proves (*).

We leave it to the reader to use (*) to prove by induction that
\begin{align*}
\schi(1)^r&=\psi_r+\frac{r(r-1)}{2}\psi_{r-2}+\cdots\\
&=\sum_{s=0}^{\lfloor r/2\rfloor}\frac{r!}{2^ss!(r-2s)!}\psi_{r-2s}.%\tag{$\dagger$}
\end{align*}
From this  we get as required that the  coefficient $a_r(\lambda)$ in the expression  $$\schi(1)^r=\sum_\lambda a_r(\lambda)\schi(\lambda)$$  is $\displaystyle\frac{r!}{2^ss!t!}\deg(\chi^\lambda)$, where $|\lambda|=t=r-2s$.
\end{proof}

Recall that induced modules for a reductive group can be realized in the algebra of regular functions of the group. Let $\lambda$ be a partition of $r$. In \cite[Prop.~1.4]{Don4} it was proved that restriction of functions induces an epimorphism $\nabla(\lambda)\to\nabla_0(\lambda)$ of $\Sp_n$-modules. Now we can form a commutative diagram as below where the vertical maps are induced by the restriction of functions $\nabla(\lambda)\to\nabla_0(\lambda)$ and the horizontal maps are evaluation at $e_1\otimes\cdots\otimes e_r$.

\begin{equation}\label{eq.sympschuriso}
\vcenter{
\xymatrix @R=30pt @C=15pt @M=6pt{
\Hom_{\GL_n}(E^{\otimes r},\nabla(\lambda))\ar[r]\ar[d]&
\nabla(\lambda)_{\varpi_r}\ar[d]\\
\Hom_{\Sp_n}(E^{\otimes r},\nabla_0(\lambda))\ar[r]&
\nabla_0(\lambda)_{\varpi_r}
}}
\end{equation}
Here $\nabla(\lambda)_{\varpi_r}$ denotes the $\varpi_r$-weight space of $\nabla(\lambda)$ with respect to $T$ and $\nabla_0(\lambda)_{\varpi_r}$ denotes the $\varpi_r$-weight space of $\nabla_0(\lambda)$ with respect to $T_0$.
\begin{lem}\label{lem.diagram}\
\begin{enumerate}[{\rm (i)}]
\item Let $M$ be a homogeneous polynomial $T$-module of degree $r$ and let $\mu\in\mb N^m$ with $|\mu|=r$. Then the $\mu$-weight space of $M$ with respect to $T$ is the same as that with respect to $T_0$.
\item Let $\lambda$ be a partition of $r$ and let $\mu\in\mb N^m$ with $|\mu|=r$. The restriction of functions $\nabla(\lambda)_\mu\to\nabla_0(\lambda)_\mu$ on the $\mu$-weight spaces with respect to $T_0$ is an isomorphism.
\item All maps in \eqref{eq.sympschuriso} are isomorphisms.
\end{enumerate}
\end{lem}
\begin{proof}
(i).\ A weight $\mu$ of $T$ vanishes on $T_0$ if and only if $\mu_i=\mu_{i'}$ for all $i\in\{1,\ldots,n\}$. So if $\mu$ and $\nu$ are weights of $T$ such that $\mu$ is polynomial, $\nu\in\mb Z^m$ , $|\mu|=|\nu|$ and $\mu|_{T_0}=\nu|_{T_0}$, then $\mu=\nu$.\\
(ii).\ Clearly $\nabla(\lambda)\to\nabla_0(\lambda)$ induces a surjection on the weight spaces for $T_0$. So it suffices to show that $\nabla(\lambda)_\mu$ and $\nabla_0(\lambda)_\mu$ have the same dimension. Note that, by (i), $\nabla(\lambda)_\mu$ is also the $\mu$-weight space with respect to $T$. Let $\mu\in\mb N^m$ with $|\mu|=r$. By \cite[4.5a]{Green} $\dim\nabla(\lambda)_\mu$ is the number of standard $\lambda$-tableaux of content $\mu$ and by \cite[\S4]{King} (or \cite[Thm.~2.3b]{Don4}) $\dim\nabla_0(\lambda)_\mu$ is the number of symplectic standard $\lambda$-tableaux of which the content $\nu$ satisfies $\ov\nu=\mu$. Here a tableau is called {\it symplectic standard} if it is standard for the ordering $1'<1<2'<2\cdots<m'<m$ of $\{1,\ldots,n\}$ and if for each $i\in\{1,\ldots,m\}$, $i$ and $i'$ only occur in the first $i$ rows. The second condition is vacuous if the content $\nu$ is in $\mb N^m$, since then $m+1,\ldots,n$ don't occur in a $\lambda$-tableau of content $\nu$. Since $\mu\in\mb N^m$, we have that $\ov\nu=\mu$ implies $\nu=\mu$ by the proof of (i). So the two dimensions are the same.\\
(iii).\ That the horizontal map in the top row of \eqref{eq.sympschuriso} is an isomorphism was pointed out before; see \eqref{eq.schuriso}. The vertical map on the right is an isomorphism by (ii). It follows that the horizontal map in the bottom row is surjective. But then it must be an isomorphism by Lemma~\ref{lem.dimension}. Now the vertical map on the left must also be an isomorphism, since it is a composite of isomorphisms.
\end{proof}

For $\lambda\in\mb N^l$ we put $S^\lambda E=S^{\lambda_1}E\otimes\cdots\otimes S^{\lambda_l}E$ and $\bigwedge{\Nts}^\lambda E=\bigwedge{\Nts}^{\lambda_1}E\otimes\cdots\otimes \bigwedge{\Nts}^{\lambda_l}E$.

\begin{lem}\label{lem.surjective}
Recall that $t=r-2s$. The following holds.
\begin{enumerate}[{\rm(i)}]
\item Let $\lambda$ be a partition of $t$. Then the canonical homomorphism
$$\Hom_{\Sp_n}(E^{\otimes r},E^{\otimes t})\otimes_{k\Sym_t}\Hom_{\Sp_n}(E^{\otimes t},\nabla_0(\lambda))\to\Hom_{\Sp_n}(E^{\otimes r},\nabla_0(\lambda)),$$
given by composition, is surjective.
\item Let $M$ be an $S(n,t)$-module. The canonical homomorphism
$$\Hom_{\Sp_n}(E^{\otimes r},E^{\otimes t})\otimes_{k\Sym_t}\Hom_{\GL_n}(E^{\otimes t},M)\to\Hom_{\Sp_n}(E^{\otimes r},M),$$
given by composition, is an isomorphism if $M$ is a direct sum of direct summands of $E^{\otimes t}$ and it is surjective if $M$ is injective.
\end{enumerate}
\end{lem}
\begin{proof}
(i). By Lemma~\ref{lem.dimension} it suffices to give a family of $\frac{r!}{s!t!2^s}\dim S(\lambda)$ elements of $\Hom_{\Sp_n}(E^{\otimes r},E^{\otimes t})\otimes_{k\Sym_t}\Hom_{\Sp_n}(E^{\otimes t},\nabla_0(\lambda))$ which is mapped to an independent family in $\Hom_{\Sp_n}(E^{\otimes r},\nabla_0(\lambda))$. As pointed out before, $\Hom_{\Sp_n}(E^{\otimes r},E^{\otimes t})$ has a basis indexed by $(t,r)$-diagrams. Let $D$ be the set of $(t,r)$-diagrams that have no horizontal edges in the top row and whose vertical edges do not cross, and let $(p_d)_{d\in D}$ be the corresponding family of basis elements in $\Hom_{\Sp_n}(E^{\otimes r},E^{\otimes t})$. Let $(u_i)_{\in I}$ be a basis of $\Hom_{\Sp_n}(E^{\otimes t},\nabla_0(\lambda))$. We have $\Hom_{\Sp_n}(E^{\otimes t},\nabla_0(\lambda))\cong S(\lambda)$ by Lemma~\ref{lem.diagram}(iii) (with $r=t$), $|D|=\frac{r!}{s!t!2^s}$ and $p_d\otimes u_i$ is mapped to $u_i\circ p_d$. So it suffices to show that the elements $u_i\circ p_d$, $d\in D$, $i\in I$, are linearly independent. So assume $\sum_{i,d}a_{id}\,u_i\circ p_d=0$ for certain $a_{id}\in k$. Consider the following diagram $d_0\in D$:
$$
\begin{xy}
(-48,2.8)*={d_0=};
(-20.4,2.8)*={\xymatrix @R=14pt @C=14pt @M=-2pt{
{\bullet}\ar@{-}[1,0]&\cdots&{\bullet}\ar@{-}[1,0]\\
{\bullet}&\cdots&{\bullet}&{\bullet}\ar@{-}[0,1]&{\bullet}&\cdots&{\bullet}\ar@{-}[0,1]&{\bullet}
}};
(-34,2.8)*=<47pt,28pt>{%
}*\frm{_\}};%
(-34,-5.2)*{\text{$t$ vertices}};
(-9.9,2.8)*=<77pt,28pt>{%
}*\frm{_\}};%
(-9.9,-5.2)*{\text{$2s$ vertices}}
\end{xy}\quad.
$$
Put $$v_0=e_1\otimes\cdots\otimes e_t\otimes e_{t+1}\otimes e_{(t+1)'}\otimes\cdots\otimes e_{t+s}\otimes e_{(t+s)'}.$$
Then we have for $d\in D$ that $p_d(v_0)=e_1\otimes\cdots\otimes e_t$ if $d=d_0$ and $0$ otherwise. It follows that $\sum_ia_{id_0}u_i(e_1\otimes\cdots\otimes e_t)=0$. By Lemma~\ref{lem.diagram}(iii) evaluation at $e_1\otimes\cdots\otimes e_t$ is injective on $\Hom_{\Sp_n}(E^{\otimes t},\nabla_0(\lambda))$, so $a_{id_0}=0$ for all $i\in I$. Since we can construct a similar vector for any other $d\in D$ it follows that $a_{id}=0$ for all $i\in I$ and $d\in D$.
\hfil\break
(ii).\ The class of $S(n,t)$-modules $M$ for which this homomorphism is an isomorphism, is closed under taking direct summands and direct sums. The same holds for the class of $S(n,t)$-modules $M$ for which this homomorphism is surjective. By \cite[Lem.~3.4(i)]{Don6} every injective $S(n,t)$-module is a direct sum of direct summands of some $S^\lambda E$, $\lambda\in\Lambda^+(n,t)$. Furthermore, $\End_{\GL_n}(E^{\otimes t})\cong k\Sym_t$. So it suffices now to show that the homomorphism is surjective if $M=S^\lambda E$, $\lambda\in\Lambda^+(n,t)$.

Denote $\Hom_{\Sp_n}(E^{\otimes r},E^{\otimes t})$ by $\mc H$ and the Schur functor $\Hom_{\GL_n}(E^{\otimes t},-)$ by $f$. Let $0\to M\to N\to P\to 0$ be a short exact sequence of $S(n,t)$-modules with a good filtration. Then we have the following diagram
%@M:distance between arrow and source and target
%@R:length of columns
%@C:length of rows
$$
\xymatrix @R=30pt @C=15pt @M=6pt{
{\mc H}\otimes_{k\Sym_t}f(M)\ar[r]\ar[d]&
{\mc H}\otimes_{k\Sym_t}f(N)\ar[r]\ar[d]&
{\mc H}\otimes_{k\Sym_t}f(P)\ar[r]\ar[d]&0\\
f_0(M)\ar[r]&
f_0(N)\ar[r]&
f_0(P)\ar[r]&0
}
$$
with rows exact, because $f$ is exact and $f_0$ is exact on modules with a good filtration. Here we have used that a $\GL_n$-module with a good $\GL_n$-filtration, also has a good $\Sp_n$-filtration; see \cite[App.~A]{Don5}. We deduce that if the homomorphism in (ii) is surjective for $N$, then it is surjective for $P$. Since the kernel of the canonical epimorphism $E^{\otimes t}\to S^\lambda E$ has a good $\GL_n$-filtration by \cite[2.1.15(ii)(b)]{Don7}, we are done.\\
\end{proof}

In the theorem below $f$ denotes the Schur functor from ${\rm mod}(S(n,t))$ to ${\rm mod}(k\Sym_t)$. Note that the second isomorphism in assertion (i) implies that when ${\rm char}\,k\ne2$ and $M$ is an injective $S(n,t)$-module, the homomorphism in Lemma~\ref{lem.surjective}(ii) is an isomorphism.

\begin{thm}\label{thm.sympschur}
Recall that $m\ge r$. The following holds.
\begin{enumerate}[{\rm(i)}]
\item For $\lambda\in\Lambda^+_0(m,r)$ we have
\begin{align*}
f_0(\nabla_0(\lambda))&\cong \widetilde{\mc S}(\lambda),\\
f_0(S^\lambda E)&\cong \widetilde{\mc M}(\lambda)\text{\quad if ${\rm char}\,k\ne2$, and}\\
f_0(\bigwedge{\Nts}^\lambda E)&\cong \mc M(\lambda)\text{\quad if ${\rm char}\,k=0$ or $>|\lambda|$.}
\end{align*}
\item Let $M$ be an $S(n,t)$-module. If $M$ is a direct sum of direct summands of $E^{\otimes t}$ or if ${\rm char}\,k\ne2$ and $M$ is injective, then
$$f_0(M)\cong {\rm Ind}^{B_r}_{k\Sym_t}\big(k_{\rm sg}\otimes f(M)\big).$$
\end{enumerate}
\end{thm}

\begin{proof}
If we give $\Hom_{\Sp_n}(E^{\otimes t},\nabla_0(\lambda))$ the $k\Sym_t$-module structure coming from the action of $\Sym_t$ on $E^{\otimes t}$ by place permutations, then the isomorphisms in \eqref{eq.sympschuriso} are $\Sym_t$-equivariant. Now Lemma~\ref{lem.surjective}(i) and the isomorphism \eqref{eq.I_siso} give us an epimorphism $I_s\otimes_{k\Sym_t}\big(k_{\rm sg}\otimes S(\lambda)\big)\to f_0(\nabla_0(\lambda))$, since $\big(I_s\otimes k_{\rm sg}\big)\otimes_{k\Sym_t}S(\lambda)\cong I_s\otimes_{k\Sym_t}\big(k_{\rm sg}\otimes S(\lambda)\big)$. The image of a nonzero homomorphism from $E^{\otimes r}$ to $\nabla_0(\lambda)$ must contain $L_0(\lambda)$ and therefore have $\lambda$ as a weight. The image of a homomorphism in $\varphi(I_{s,1})$ does not have $\lambda$ as a weight, since $\varphi(I_{s,1})$ has a basis of homomorphisms whose image lies is a submodule of $E^{\otimes t}$ which is isomorphic to $E^{\otimes(t-2)}$. So, by \eqref{eq.Z_siso} and the definition of $\widetilde{S}(\lambda)$, we obtain an epimorphism $\widetilde{S}(\lambda)\to f_0(\nabla_0(\lambda))$. By Lemma~\ref{lem.dimension} this must be an isomorphism.

Let $M$ be an $S(n,t)$-module. Lemma~\ref{lem.surjective}(ii) and the isomorphism $\varphi$ give us a homomorphism
$${\rm Ind}^{B_r}_{k\Sym_t}\big(k_{\rm sg}\otimes f(M)\big)\to f_0(M)\eqno(*)$$
which is an isomorphism if $M$ is a direct sum of direct summands of $E^{\otimes t}$ and surjective for $M$ injective. Let $\lambda\in\Lambda_0(n,r)$ be a partition of $t$. Then we obtain an epimorphism $\widetilde{\mc M}(\lambda)\to f_0(S^\lambda E)$ and a homomorphism ${\mc M}(\lambda)\to f_0(\bigwedge{\Nts}^\lambda E)$, since $f(S^\lambda E)=M(\lambda)$ and $f(\bigwedge{\Nts}^\lambda E)=k_{\rm sg}\otimes M(\lambda)$ by \cite[Lemma~3.5]{Don6}. If ${\rm char}\,k=0$ or $>t$, then $S(n,t)$ is semisimple, so every $S(n,t)$-module is a direct sum of direct summands of $E^{\otimes t}$ and (*) is an isomorphism for every $S(n,t)$-module $M$. In particular, we have the third isomorphism in (i).

Now assume that ${\rm char}\,k\ne2$. We want to show that the epimorphism $\widetilde{\mc M}\to f_0(S^\lambda E)$ is an isomorphism. Since (*) is an isomorphism if ${\rm char}\,k=0$ it suffices to show that the dimensions of $f_0(S^\lambda E)$ and $\widetilde{\mc M}$ are independent of the characteristic. The dimension of $f_0(S^\lambda E)$ is independent of the characteristic, since, by \cite[Prop.~A.2.2(ii)]{Don7}, it only depends on the formal characters of the $\Sp_n$-modules $E^{\otimes r}$ and $S^\lambda E$ (and these are independent of the characteristic). That $\widetilde{\mc M}$ has dimension independent of the characteristic follows from Proposition~\ref{prop.filtration}, the fact that $M(\lambda)$ is self dual and the following fact, the proof of which we leave to the reader.

{\it Assume ${\rm char}\,k\ne2$. Let $G$ be a finite group with a (possibly trivial) sign homomorphism ${\rm sg}:G\to\{\pm1\}$, let $V$ be a permutation module for $G$ over $k$ with $G$-stable basis $S$. Then the dimension of $(k_{\rm sg}\otimes V)^G$ is equal to the number of $G$-orbits in $S$ for which one (and therefore each) stabilizer is contained in $\Ker(\rm sg)$.}
%We apply this to $M(\lambda)$ which is considered as a permutation module for the various hyperoctahedral subgroups $H_i$.

We have now proved the second isomorphism in (i) and we have also proved (ii), since every injective $S(n,t)$-module is a direct sum of direct summands of some $S^\lambda E$, $\lambda\in\Lambda^+(n,t)$.
\end{proof}
Note that it follows from Theorem~\ref{thm.sympschur}(i) that $f_0$ maps good filtrations to twisted Specht filtrations.

For $\lambda\in\Lambda_0^+(m,r)$ we denote the indecomposable tilting module for $\Sp_n$ of highest weight $\lambda$ by $T_0(\lambda)$ and for $\lambda$ $p$-regular we denote the projective cover of the irreducible $B_r$-module $\widetilde{\mc D}(\lambda)$ by $\widetilde{\mc P}(\lambda)$.
\begin{prop}\label{prop.sympmult}
Let $\lambda\in\Lambda_0^+(m,r)$. Then $T_0(\lambda)$ is a direct summand of the $\Sp_n$-module $E^{\otimes r}$ if and only if $\lambda$ is $p$-regular and $\lambda\ne\emptyset$ in case $r$ is even $\ge2$ and $\delta=0$. Now assume that $\lambda$ satisfies these conditions. Then
\begin{enumerate}[{\rm(i)}]
\item $f_0(T_0(\lambda))=\widetilde{\mc P}(\lambda)$.
\item The multiplicity of $T_0(\lambda)$ in $E^{\otimes r}$ is $\dim\widetilde{\mc D}(\lambda)$.
\item The decomposition number $[\widetilde{\mc S}(\mu):\widetilde{\mc D}(\lambda)]$ is equal to the $\Delta$-filtration multiplicity $(T_0(\lambda):\Delta_0(\mu))$ and to the $\nabla$-filtration multiplicity $(T_0(\lambda):\nabla_0(\mu))$.
\end{enumerate}
\end{prop}

\begin{proof}
Let $\Omega$ be the set of all partitions satisfying the stated conditions. The symplectic Schur functor $f_0$ induces a category equivalence between the direct sums of direct summands of the $\Sp_n$-module $E^{\otimes r}$ and the projective $B_r$-modules; see e.g. \cite[Prop~2.1(c)]{ARS}. Clearly, the number of isomorphism classes of indecomposable $B_r$-projectives is equal to $|\Omega|$. So, to prove the first assertion, it suffices to show that for each $\lambda\in\Omega$, $T_0(\lambda)$ is a direct summand of $E^{\otimes r}$. By Lemma~\ref{lem.directsummand} we may assume that $t=r$. The indecomposable $\GL_n$ tilting module $T(\lambda)$ is a direct summand of $E^{\otimes r}$, for example by \cite[Sect.~4.3, (1) and (4)]{Don7}.  Moreover, as an $\Sp_n$-module, $T(\lambda)$ is also a tilting module and has unique highest weight $\lambda$, and $\lambda$ occurs with multiplicity $1$. Thus we have $T(\lambda)\cong T_0(\lambda)\oplus Y$, where $Y$ is a direct sum of indecomposable tilting modules for $\Sp_n$, of weight less than $\lambda$. In particular, $T_0(\lambda)$ occurs as a component of $E^{\otimes r}$.

Now let $\lambda\in\Omega$. By Theorem~\ref{thm.sympschur}(i) we have that $f_0(T_0(\lambda))$ surjects onto $f_0(\nabla_0(\lambda))=\widetilde{\mc S}(\lambda)$. But $\widetilde{\mc S}(\lambda)$ surjects onto $\widetilde{\mc D}(\lambda)$. This proves (i), and (ii) is now also clear, since this multiplicity (as an indecomposable direct summand) is equal to the multiplicity of $\widetilde{\mc P}(\lambda)$ in $B_r$. We have $g_0(f_0(M))\cong M$ canonically for $M=E^{\otimes r}$ and therefore also for $M=T_0(\lambda)$. By \eqref{eq.adjointiso} we have $\Hom_{\Sp_n}(T_0(\lambda),M)\cong\Hom_{B_r}(\widetilde{\mc P}(\lambda),f_0(M))$ for every $S_0(n,r)$-module $M$. So $[\widetilde{\mc S}(\mu):\widetilde{\mc D}(\lambda)]=\dim\Hom_{B_r}(\widetilde{\mc P}(\lambda),\widetilde{\mc S}(\mu))=\dim\Hom_{\Sp_n}(T_0(\lambda),\nabla_0(\mu))=(T_0(\lambda):\Delta_0(\mu))$. The equality $(T_0(\lambda):\Delta_0(\mu))=(T_0(\lambda):\nabla_0(\mu))$ follows from the fact that both multiplicities are equal to the coefficient of the Weyl character $\schi(\mu)$ in ${\rm ch}\,T_0(\lambda)$.
\end{proof}

\begin{comment}
The next corollary is a result of Littlewood; it shows that Theorem~\ref{thm.sympschur} is also of interest in characteristic 0.

\begin{cornn}[\cite{Lit}]\label{cor.littlewood}
The multiplicity of $\nabla_0(\mu)$ in $\nabla(\lambda)$ is $\sum_{\nu: \nu' \text{\,even}}c^\lambda_{\mu\,\nu}$.
\end{cornn}
\begin{proof}
This multiplicity is equal to the dimension of $\Hom_{\Sp_n}(\Delta_0(\nu),\nabla(\lambda))$ which only depends on the formal characters of the $\Sp_n$-modules $\Delta_0(\nu)$ and $\nabla(\lambda)$. Since these formal characters are independent of the characteristic, we may assume that $k$ has characteristic $0$. By Theorem~\ref{thm.sympschur} this multiplicity is equal to the multiplicity of $\mc S(\mu')$ in ${\rm Ind}^{B_r}_{k\Sym_t}S(\lambda')$. The result now follows from Proposition~\ref{prop.bla} and the fact that $c^{\lambda'}_{\mu'\,\nu}=c^{\lambda}_{\mu\,\nu'}$.
\end{proof}
%In F.\ Gavarini, {\it A Brauer algebra-theoretic proof of Littlewood's restriction rules.}, J. Algebra {\bf 212} (1999), no. 1, 240–271, one can find another Brauer algebra theoretic proof of Littlewood's theorem. But I think our approach with the symplectic Schur functor is more elegant. Unfortunately we would only get it under the assumption that $|\lambda|,|\mu|\le m$. So one would need an elementary argument to show that the multiplicity is independent of $m$ to get the result for any partitions $\lambda$ and $\mu$ of length $\le m$.
\end{comment}

We now combine Proposition~\ref{prop.sympmult} with a result of Adamovich and Rybnikov. For this we need the following notation. Let $m'$ be a positive integer and let $\lambda$ be a partition with $l(\lambda)\le m$ and $l(\lambda')=\lambda_1\le m'$, that is, a partition of which the diagram fits into an $m\times m'$-rectangle. Here $\lambda'$ denotes the transpose of $\lambda$. Then we define
$$\lambda^\dagger=(m-\lambda'_{m'},m-\lambda'_{m'-1},\ldots,m-\lambda'_1).$$
So $\lambda^\dagger$ is the transpose of the complement of $\lambda$ in the $m\times m'$-rectangle. In particular, $l(\lambda^\dagger)\le m'$ and $\lambda^\dagger_1\le m$.

\begin{cornn}
Let $\lambda,\mu\in\Lambda_0^+(r,r)$ with $\lambda$ $p$-regular. Assume that $\lambda_1,\mu_1\le m'$. Then we have the equality of decomposition numbers
$$[\widetilde{\mc S}(\mu):\widetilde{\mc D}(\lambda)]=[\nabla_0'(\lambda^\dagger):L_0'(\mu^\dagger)]\,,$$
where $\nabla'_0$ and $L'_0$ denote induced and irreducible modules for $\Sp_{2m'}$.
\end{cornn}

\begin{proof}
This follows immediately from Proposition~\ref{prop.sympmult}(iii) and \cite[Cor~2.4]{AdRyb}.
\end{proof}

\begin{rems}\label{rems.sympschur}
1.\ Let $f_0^t$ denote the symplectic Schur functor from ${\rm mod}(S_0(n,t))$ to ${\rm mod}(B_t)$ and let $M$ be an $\Sp_n$-module which has a filtration with sections isomorphic to some $\nabla_0(\lambda)$, $\lambda$ a partition of $t$. Then $$f_0(M)\cong Z_s\otimes_{k\Sym_t}f_0^t(M).$$
This is shown as follows. First we construct a homomorphism
$$\Hom_{\Sp_n}(E^{\otimes r},E^{\otimes t})\otimes_{k\Sym_t}f_0^t(M)\to f_0(M)$$
by means of function composition. Then we form, for a short exact sequence $0\to M\to N\to P\to 0$ of $\Sp_n$-modules of the above type, a diagram as in the proof of Lemma~\ref{lem.surjective}(ii) with $f$ replaced by $f_0^t$ and deduce that if the homomorphism is surjective for $M$ and $P$, then also for $N$. We then obtain surjectivity by induction on the length of a good filtration. Then we factor out $\varphi(I_{s,1})$ as in the proof of the first isomorphism of Theorem~\ref{thm.sympschur}(i) and finish by showing that the dimensions are equal.\\
2.\ Let $M$ be an $S(n,r)$-module. Put $\pi_r=\{\lambda\in\Lambda^+_0(m,r)\,|\,|\lambda|<r\}$ and let $N=O_{\pi_r,0}(M)$ be the largest $\Sp_n$-submodule of $M$ which belongs to $\pi_r$, i.e. which has only composition factors $L(\lambda)$, $\lambda\in\pi_r$. By \cite[Prop.~A2.2(v), Lem.~A3.1]{Don7} $N$ has a filtration with sections $\nabla_0(\lambda)$, $\lambda\in\pi_r$, and $M/N$ has a filtration with sections $\nabla_0(\lambda)$, $\lambda$ a partition of $r$. Note that $M/N=\nabla_0(\lambda)$ if $M=\nabla(\lambda)$. Now we can form the diagram
\begin{equation*}
\xymatrix @R=30pt @C=15pt @M=6pt{
f(M)\ar[r]\ar[d]&
M_{\varpi_r}\ar[d]\\
f_0(M/N)\ar[r]&
(M/N)_{\varpi_r}
}
\end{equation*}
in the same way as \eqref{eq.sympschuriso} and by a proof very similar to that of Lemma~\ref{lem.diagram}(iii) we show that all maps are isomorphisms. For the isomorphism $f(M)\stackrel{\sim}{\to}f_0(M/N)$ to be $B_r$-equivariant one needs to twist $f(M)$ with the sign.\\
3.\ It is easy to see that the canonical homomorphism ${\mc M}(\lambda)\to f_0(\bigwedge{\Nts}^\lambda E)$ constructed in the proof of Theorem~\ref{thm.sympschur} is not always an isomorphism. Assume that $r\ge 2p$ and take $\lambda=(r)$. Then ${\mc M}(\lambda)=B_r\otimes_{k\Sym_r}k$ is a cyclic $B_r$-module with generator $1\otimes 1$ which is not killed by any diagram in $B_r$. But its image in $f_0(\bigwedge{\Nts}^\lambda E)$ is the canonical projection $P:E^{\otimes r}\to\bigwedge^r\Nts E$ which is killed by any diagram with at least $p$ horizontal edges in a row, since the $p^{\rm th}$ power of the symplectic invariant is zero in the exterior algebra $\bigwedge E$.\\
4.\ We have the adjoint isomorphism
$$\Hom_{S_0(n,r)}(V\otimes_{k\Sym_t}E^{\otimes t},E^{\otimes r})\cong\Hom_{k\Sym_t}(V,\Hom_{S_0(n,r)}(E^{\otimes t},E^{\otimes r}))$$
for every $\Sym_t$-module $V$; see e.g. \cite[Thm.~2.11]{Rot}. From this we deduce that for every $S(n,t)$ module $M$ with $g(f(M))\cong M$ canonically, we have $f_0(M^*)\cong\Hom_{k\Sym_t}(f(M),k_{\rm sg}\otimes I_s)$ and $f_0(M)\cong\Hom_{k\Sym_t}(k_{\rm sg}\otimes I_s^*,f(M))$. In particular we obtain for a partition $\lambda$ of $t$, $f_0(\bigwedge{\Nts}^\lambda E)\cong\Hom_{k\Sym_t}(M(\lambda),I_s)$. Unfortunately, we have been unable to make effective use of this isomorphism. %It is, for example, not clear whether $g_0(\Hom_{k\Sym_t}(M(\lambda),I_s))$ is isomorphic to $\bigwedge{\Nts}^\lambda E$.
\end{rems}

\begin{lem}\label{lem.tensor}
Let $M$ be an $\Sp_n$-module. Then the canonical homomorphism
$$E^{\otimes r}\otimes_{B_r}\Hom_{\Sp_n}(E^{\otimes r},M)\to M$$ given by function application is an isomorphism if $M$ is a direct summand of $E^{\otimes r}$ or if $r$ is even $\ge4$ and $M=k$.
\end{lem}
\begin{proof}
That the canonical homomorphism is an isomorphism under the first condition is obvious, since $\End_{\Sp_n}(E^{\otimes r})\cong B_r$. So assume $r$ is even $\ge4$ and $M=k$. Since the homomorphism is always surjective and
$$E^{\otimes r}\otimes_{B_r}\Hom_{\Sp_n}(E^{\otimes r},k)\cong\Hom_{B_r}(\Hom_{\Sp_n}(E^{\otimes r},k),E^{\otimes r})^*$$
by \cite[Lemma~3.60]{Rot}, it suffices to show that $\Hom_{B_r}(\Hom_{\Sp_n}(E^{\otimes r},k),E^{\otimes r})$ is one-dimensional. Recall that $\Hom_{\Sp_n}(E^{\otimes r},k)$ is a left $B_r$-module by means of the standard anti-automorphism $\iota$ of $B_r$. It has a basis indexed by $(0,r)$-diagrams and it is generated as a $k\Sym_r$-module by the homomorphism $P$ corresponding to the $(0,r)$-diagram
$$
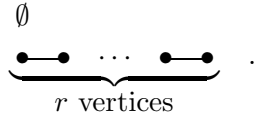
\begin{xy}
(-20.4,2.8)*={\xymatrix @R=14pt @C=14pt @M=-2pt{
\emptyset\\
{\bullet}\ar@{-}[0,1]&{\bullet}&\cdots&{\bullet}\ar@{-}[0,1]&{\bullet}
}};
(-28.7,2)*=<78pt,28pt>{%
}*\frm{_\}};%
(-28.7,-5.9)*{\text{$r$ vertices}};
\end{xy}\quad.
$$
It follows that any $B_r$-homomorphism from $\Hom_{\Sp_n}(E^{\otimes r},k)$ to $E^{\otimes r}$ is determined by its image of $P$. One easily checks that $P\circ\iota(d)=\pm P$, where $d\in B_r$ is given by
$$
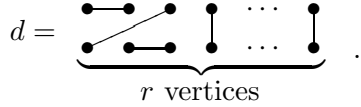
\begin{xy}
(-48,2.8)*={d=};
(-20.4,2.8)*={\xymatrix @R=14pt @C=14pt @M=-2pt{
{\bullet}\ar@{-}[0,1]&{\bullet}&{\bullet}\ar@{-}[1,-2]&{\bullet}\ar@{-}[1,0]&\cdots&{\bullet}\ar@{-}[1,0]\\
{\bullet}&{\bullet}\ar@{-}[0,1]&{\bullet}&{\bullet}&\cdots&{\bullet}
}};
(-25.9,2.8)*=<92pt,28pt>{%
}*\frm{_\}};%
(-26,-5.2)*{\text{$r$ vertices}};
\end{xy}\quad.
$$
Therefore the image of $P$ must lie in $d\cdot E^{\otimes r}=ku\otimes E^{\otimes(r-2)}$, where $u$ is the invariant $\sum_{i=1}^n\se_ie_i\otimes e_{i'}$. Similarly we find that it must lie in $E^{\otimes i}\otimes ku\otimes E^{\otimes (r-i-2)}$ for any even integer $i$ with $0\le i\le r-2$. We conclude that the image of $P$ under any $B_r$-homomorphism from $\Hom_{\Sp_n}(E^{\otimes r},k)$ to $E^{\otimes r}$ must be a scalar multiple of $u^{\otimes r/2}$.
\end{proof}

The symplectic Schur coalgebra is $A_0(n,r)=O_{\Lambda_0^+(m,r)}(k[\Sp_n])$ and $S_0(n,r)=A_0(n,r)^*$, where the left action of $\Sp_n$ on $k[\Sp_n]$ comes from right multiplication in $\Sp_n$; see \cite{Don1} for the generalities. Recall that $E^{\otimes r}$ is self-dual as an $\Sp_n$-module and as a $B_r$-module. It follows that
$$f_0(A_0(n,r))=\Hom_{\Sp_n}(E^{\otimes r},S_0(n,r)^*)\cong\Hom_{\Sp_n}(S_0(n,r),E^{\otimes r})\cong E^{\otimes r}\text{\ and}$$
\begin{equation}\label{eq.g_0}
g_0(E^{\otimes r})=E^{\otimes r}\otimes_{B_r}E^{\otimes r}\cong\End_{B_r}(E^{\otimes r})^*=S_0(n,r)^*\cong A_0(n,r).
\end{equation}
In the proposition below $g$ denotes the inverse Schur functor from ${\rm mod}(k\Sym_t)$ to ${\rm mod}(S(n,t))$.
\begin{prop}\label{prop.sympyoung}\
\begin{enumerate}[{\rm(i)}]
\item If $n=0$ in $k$ and $t=0$, assume $r\ge4$. Then we have
$$g_0({\rm Ind}^{B_r}_{k\Sym_t} V)\cong g(k_{\rm sg}\otimes V)$$
as $\Sp_n$-modules, for every $k\Sym_t$-module $V$.
\item Let $\lambda\in\Lambda^+_0(m,r)$. If $\lambda=\emptyset$ and $n=0$ in $k$, then assume $r\ge4$. Then $g_0(\widetilde{\mc M}(\lambda))\cong S^\lambda E$ and if ${\rm char}\,k\ne2$, then $g_0({\mc M}(\lambda))\cong \bigwedge{\Nts}^\lambda E$.
%We have $g_0(f_0(M))\cong M$ canonically for $M=E^{\otimes r}$, $M=A_0(n,r)$ and, if ${\rm char}\,k\ne2$, for $M=S^\lambda E$, $\lambda\in\Lambda^+_0(m,r)$; we have $f_0(g_0(V))\cong V$ canonically for $V=B_r$, $V=E^{\otimes r}$ and, if ${\rm char}\,k\ne2$, for $V=\widetilde{\mc M}(\lambda)$, $\lambda\in\Lambda^+_0(m,r)$. Here we require in both cases $r\ge4$ if $\lambda=\emptyset$ and $m=0$ in $k$.
\item Let $\lambda\in\Lambda^+_0(m,r)$.
      The $\Sp_n$-module $S^\lambda E$ has a unique indecomposable summand $J(\lambda)$ in which $\nabla_0(\lambda)$ has filtration multiplicity $>0$ and this multiplicity is equal to $1$.
      If ${\rm char}\,k\ne2$, then every summand of $\widetilde{\mc M}(\lambda)$ has a twisted Specht filtration and $f_0(J(\lambda))\cong\widetilde{\mc Y}(\lambda)$.
\end{enumerate}
\end{prop}
\begin{proof}
(i).\ Since ${\rm Ind}^{B_r}_{k\Sym_t} V\cong\big(\Hom_{\Sp_n}(E^{\otimes r},E^{\otimes t})\otimes k_{\rm sg}\big)\otimes_{k\Sym_t}V$ which is isomorphic to $\Hom_{\Sp_n}(E^{\otimes r},E^{\otimes t})\otimes_{k\Sym_t}\big(k_{\rm sg}\otimes V\big)$, this follows from Lemmas~\ref{lem.directsummand} and \ref{lem.tensor} applied to $E^{\otimes t}$.\\
(ii).\ By (i) (with $t=|\lambda|$) we have $g_0(\widetilde{\mc M}(\lambda))\cong g(M(\lambda))$ and $g_0({\mc M}(\lambda))\cong g(k_{\rm sg}\otimes M(\lambda))$. One easily verifies that $g(M(\lambda))\cong S^\lambda E$ and, in case ${\rm char}\,k\ne2$, $g(k_{\rm sg}\otimes M(\lambda))\cong\bigwedge{\Nts}^\lambda E$.\\
(iii).\ Put $t=|\lambda|$. Then $r-t=2s$ is even. The filtration multiplicity of $\nabla(\lambda)$ in $S^\lambda E$ is 1 and if $\nabla_0(\nu)$ has filtration multiplicity $>0$ in $\nabla(\mu)$, then either $\nu=\mu$ and the multiplicity is 1 or $|\nu|<|\mu|$ as one can easily deduce from Lemma~\ref{lem.diagram}(ii). We conclude that the filtration multiplicity of $\nabla_0(\lambda)$ in $S^\lambda E$ is 1. A direct summand of a module with a good filtration has a good filtration. So, by the Krull-Schmidt theorem, there is a unique indecomposable summand $J(\lambda)$ in which $\nabla_0(\lambda)$ has filtration multiplicity $>0$. This proves the first assertion. Now assume ${\rm char}\,k\ne2$. If $\lambda=\emptyset$, then $S^\lambda E=k$, $Z_{r/2}=I_{r/2}$ and $\widetilde{\mc S}(\lambda)=\widetilde{\mc M}(\lambda)=\widetilde{\mc Y}(\lambda)=I_{r/2}$ by \cite[Cor.~3.2]{HarPag} and the assertion is obvious. Now assume $\lambda\ne\emptyset$. By (ii) and Theorem~\ref{thm.sympschur}(i) we have $g_0(f_0(M))\cong M$ canonically for every direct summand of $S^\lambda E$ and $f_0(g_0(V))\cong V$ canonically for every direct summand of $\widetilde{\mc M}(\lambda)$. In particular, every direct summand of $\widetilde{\mc M}(\lambda)$ has a twisted Specht filtration.

Let $I(\lambda)\subseteq S^\lambda E$ be the $S(n,t)$-injective hull of $\nabla(\lambda)$. By \cite[3.6]{Don6} we have $f(I(\lambda))=Y(\lambda)$. Put $\pi=\pi_t=\{\mu\in\Lambda^+_0(m,r)\,|\,|\mu|<t\}$. By Remarks~\ref{rems.sympschur}, 1 and 2 we have $f_0(I(\lambda)/O_\pi(I(\lambda)))\cong Z_s\otimes_{k\Sym_t}\big(k_{\rm sg}\otimes Y(\lambda)\big)$. By \cite[Prop.~3.1]{HarPag} $Z_s\otimes_{k\Sym_t}\big(k_{\rm sg}\otimes Y(\lambda)\big)$ is indecomposable. Since $I(\lambda)/O_\pi(I(\lambda))$ has a good filtration, it must also be indecomposable. Now write $I(\lambda)=\bigoplus_{i=1}^l J_i$ with each $J_i$ an indecomposable $\Sp_n$-module. Then $I(\lambda)/O_\pi(I(\lambda))\cong\bigoplus_{i=1}^l J_i/O_\pi(J_i)$. So there is a unique $j$ such that $J_j/O_\pi(J_j)\cong I(\lambda)/O_\pi(I(\lambda))$ and $J_i\subseteq O_\pi(I(\lambda))$ for all $i\ne j$. Clearly we must have $J_j\cong J(\lambda)$. Furthermore, since the kernel of $J(\lambda)\to J(\lambda)/O_\pi(J(\lambda))$ has a good filtration, we have that $f_0(J(\lambda))$ surjects onto $Z_s\otimes_{k\Sym_t}\big(k_{\rm sg}\otimes Y(\lambda)\big)$. So $f_0(J(\lambda))\cong\widetilde{\mc Y}(\lambda)$.
\end{proof}

\begin{rems}\label{rems.g_0}
1.\ Drop the assumption that $\delta=-n$. Assume that ${\rm char}\,k\ne2,3$ and also $r\ne2,4$ in case $\delta=0$. Then it follows from \cite[Thm~1.6]{ErdSa} and the results in \cite{HarPag} that a direct summand of a module with a Specht or twisted Specht filtration has again such a filtration. In particular this applies to the Young modules for the Brauer algebra and their twisted versions.\\
2.\ The class of $S_0(n,r)$-modules $M$ for which $g_0(f_0(M))\cong M$ canonically, is closed under taking direct summands and direct sums. In particular it contains the injective $S_0(n,r)$-modules, since, by \eqref{eq.g_0}, it contains $A_0(n,r)$. For the same reason the class of $B_r$-modules $V$ for which $f_0(g_0(V))\cong V$ canonically, contains the projective $B_r$-modules.\\
3.\ For $\lambda\in X_0$ let $I_\lambda$ be the $\lambda$-weight space of $A_0(n,r)$ for the action of $T_0$ which comes from left multiplication in $G$.
One shows as in \cite[2.4]{Don2} that $\Hom(-,I_\lambda)\cong M\mapsto (M_\lambda)^*$ as functors on ${\rm mod}(S_0(n,r))$. In particular, $I_\lambda$ is an injective $S_0(n,r)$-module. Furthermore, $A_0(n,r)=\bigoplus_{\lambda\in\Lambda_0(m,r)}I_\lambda$ and $I_\lambda\cong I_{w(\lambda)}$ for every $w$ in the Weyl group $W$. So every injective $S_0(n,r)$-module is a direct sum of direct summands of $I_\lambda$'s, $\lambda\in\Lambda^+_0(m,r)$. Since the nondegenerate bilinear form of $E^{\otimes r}$ is also nondegenerate on the weight spaces of $E^{\otimes r}$ for $T_0$, we have that those weight spaces are self-dual $B_r$-modules. So $f_0(I_\lambda)=(E^{\otimes r})_{\lambda,0}$, the $\lambda$-weight space of $E^{\otimes r}$ for $T_0$. Note that $g_0((E^{\otimes r})_{\lambda,0})=I_\lambda$ by the preceding remark. By Lemma~\ref{lem.diagram}(i) we have $(E^{\otimes r})_{\lambda,0}\cong M(\lambda)$ if $\lambda$ is a partition of $r$. Here the diagrams with a horizontal edge act as $0$ on $M(\lambda)$.
By \cite[Thm 9.51]{Rot} (with $S=C=k$), we have
$$L^ig_0(V)=\Tor_i^{B_r}(E^{\otimes r},V)\cong\Ext^i_{B_r}(V,E^{\otimes r})^*$$
as vector spaces, for every $B_r$-module $V$. Here $L^ig_0$ denotes the $i^{\rm th}$ left derived functor of $g_0$.
%Other argument (closer to \cite{Don8}):
%If $M$ is an injective $S_0(n,r)$-module and $V$ is a $B_r$-module, then
%$$\Hom_{S_0(n,r)}(L^ig_0(V),M)\cong\Ext^i_{B_r}(V,f_0(M)),$$
%by \cite[Thm.~11.54]{Rot}, since the spectral sequence degenerates ($M$ is injective).
So to show that, under certain conditions on $m$, $r$ and the field $k$, $L^1g_0(V)=0$ whenever $V$ has a twisted Specht filtration, it suffices to show that $\Ext^1_{B_r}(\widetilde{\mc S}(\lambda),(E^{\otimes r})_{\mu,0})=0$ for all $\lambda,\mu\in\Lambda^+_0(m,r)$. This would mean that $g_0$ maps exact sequences of modules with a twisted Specht filtration to exact sequences and it would imply as in the proof of \cite[Prop.~10.6]{Don8}
%Here one needs that $g_0((E^{\otimes r})_{\lambda,0})=I_\lambda$
that $g_0(\widetilde{\mc S}(\lambda))=\nabla_0(\lambda)$. This, in turn, would imply that $g_0$ maps $B_r$-modules with a twisted Specht filtration to $\Sp_n$-modules with a good filtration.
%and it would give an alternative proof for the proof in \cite[Thm.~4.1]{HarPag} of the well-definedness of (twisted) Specht filtration multiplicities.
\end{rems}

\section{The Jantzen sum formula and a block result for $S_0(n,r)$}\label{s.jantzen}

The notation is as in the previous section. In this section we assume that the field $k$ is of positive characteristic $p>2$. Let $u$ be the unique integer satisfying $-p/2<n-up<p/2$. Put $$\delta=-(n-up).$$
Throughout this section we will be working with the root system of type $C_m$ in the vector space $\mb R^m$ endowed with the standard inner product
$$\la x,y\ra=\sum_{i=1}^mx_iy_i.$$
The weight lattice is identified with $\mb Z^m$. The positive roots are $2\e_i$, $1\le i\le m$, and $\e_i\pm\e_j$, $1\le i<j\le m$. Recall that the Weyl group $W(C_m)$ acts by signed permutations and that the {\it dot action} of $W(C_m)$ on $\mb R^m$ is given by
$$w\cdot x=w(x+\rho)-\rho,$$
where
$$\rho=(m,\ldots,2,1).$$
Note that $\mb Z^m$ is stable under the dot action.

Let $G$ be a reductive group and let $X$ be the group of weights relative to a fixed maximal torus. Jantzen has defined for every Weyl module $\Delta(\lambda)$ of $G$ a descending filtration $\Delta(\lambda)=\Delta(\lambda)^0\supseteq\Delta(\lambda)^1\supseteq\cdots$ such that $\Delta(\lambda)/\Delta(\lambda)^1\cong L(\lambda)$ and $\Delta(\lambda)^i=0$ for $i$ big enough. The Jantzen sum formula \cite[II.8.19]{Jan} relates the formal characters of the $\Delta(\lambda)^i$ with the Weyl characters $\chi(\mu)$:
\begin{equation}\label{eq.jantzensum}
\sum_{i>0}{\rm ch}\,\Delta(\lambda)^i=\sum\nu_p(lp)\chi(s_{\alpha,l}\cdot\lambda)\ ,
\end{equation}
where the sum on the right is over all pairs $(\alpha,l)$, with $l$ an integer $\ge1$ and $\alpha$ a positive root such that $\la\lambda+\rho,\alpha^\vee\ra-lp>0$. Here $\nu_p$ is the $p$-adic valuation, $\alpha^\vee=\frac{2}{\la\alpha,\alpha\ra}\alpha$ and $s_{\alpha,l}$ is the affine reflection of $\mb R\otimes_{\mb Z}X$ defined by $s_{\alpha,l}(x)=x-a\alpha$, where $a=\la x,\alpha^\vee\ra-lp$. It should be noted that the $s_{\alpha,l}\cdot\lambda$ are in general not dominant. But if $\chi(s_{\alpha,l}\cdot\lambda)\ne0$, then it can be written as $\pm\chi(\mu)$ for some dominant weight $\mu$ using \cite[II.5.9(1)]{Jan}. The $\chi(\mu)$, $\mu$ a dominant weight, form a $\mb Z$-basis of $(\mb Z X)^W$, where $W$ denotes the Weyl group.

We are now going to prove a strengthened version of the Jantzen sum formula for the symplectic group in a certain generic situation and deduce from this formula a block result for $S_0(n,r)$. In Section~\ref{s.blocks} we will then deduce from this the block result \cite[Thm~4.2]{CdVM2} for the Brauer algebra in characteristic $0$. Following Cox, De Visscher and Martin we define $\hat\rho\in\frac{1}{2}\mb Z^r$ by
$$\hat\rho=\big(-\frac{\delta}{2},-\frac{\delta}{2}-1,\ldots,-\frac{\delta}{2}-(r-1)\big)$$
and the {\it star action} (called \lq\lq dot action" in \cite{CdVM2}) of $W(C_r)$ on $\mb R^r$ by
$$w\star x=w(x+\hat\rho)-\hat\rho.$$
Note again that $\mb Z^r$ is stable under this action. The Weyl group $W(C_r)$ contains the Weyl group $W(D_r)$ of type $D_r$ which consists of those signed permutations that involve an even number of sign changes. When $m\ge r$ we will use the convention that $r$-tuples can also be considered as $m$-tuples by extending them with zeros. For a partition $\lambda$, we denote its length by $l(\lambda)$ and we denote the transposed partition by $\lambda'$.
\begin{thm}\label{thm.sumformula}
Let $\lambda$ be a partition with $l(\lambda),l(\lambda')\le r$. If $n+2r<p$, then $\Delta_0(\lambda)$ is irreducible. If $|\delta|+2r<p/2$ and $u\ne0$, then $m>r$ and we have
\begin{equation}\label{eq.sumformula}
\sum_{i>0}{\rm ch}\,\Delta_0(\lambda)^i=
\nu_p(up)\sum\schi(s_\alpha\star\lambda)\ ,
\end{equation}
where the sum on the right is over all positive roots $\alpha=\e_i+\e_j$, $1\le i<j\le r$, with $\la\lambda+\hat\rho,\alpha^\vee\ra>0$.
\end{thm}
% Since we are interested in the situation that
\begin{proof}
The first assertion follows immediately from the fact that then the sum on the right in \eqref{eq.jantzensum} is empty. Now we assume that $|\delta|+2r<p/2$ and $u\ne0$. Clearly $m>r$.
%If $m\le r$, then $n\le 2r<p/2$. So we would have $\delta=-n$ and $u=0$.
The idea is to deduce \eqref{eq.sumformula} from \eqref{eq.jantzensum} by showing that certain summands on the right hand side of \eqref{eq.jantzensum} may be omitted.
In the proof below we will need some extra notation. For $i\in\{1,\ldots,m\}$, we put $i'=m+1-i$. Note that this differs from the notation in Subsection~\ref{ss.sympschur}. For $x\in\mb R^m$ we define $\breve{x}$ to be the reversed tuple of $x$. So $\breve{x}_i=x_{i'}$. Note that $\breve{\rho}_i=i$ and that the first $m-r$ entries of $(\lambda+\rho)\breve{\ }$ form the interval $\{1,2\ldots,m-r\}$. Furthermore, $\alpha$ will always denote a positive root in the root system of type $C_m$. We will use the following fact. Let $\lambda\in\mb Z^m$. Then
\begin{align*}
&\schi(\lambda)\ne0\text{\quad if and only if}\\
&(\lambda+\rho)_i\ne0\text{\quad for all\ }i\in\{1,\ldots,m\}\text{\quad and}\\
&(\lambda+\rho)_i\ne\pm(\lambda+\rho)_j\text{\quad for all\ }i,j\in\{1,\ldots,m\}\text{\ with\ }i\ne j.
\end{align*}
The proof will consist of three lemma's.

\begin{lemgl}\label{lem.e_i-e_j}
Assume $\alpha=\e_i-\e_j$, $1\le i<j\le m$ and $\la\lambda+\rho,\alpha^\vee\ra=a+lp$, $a,l>0$. Then $\schi(s_{\alpha,l}\cdot\lambda)=0$.
\end{lemgl}

\begin{proof}
First we redefine $i$ and $j$ by replacing $(i,j)$ by $(j',i')$. So $1\le i<j\le m$ and $\alpha=\alpha^\vee=\e_{j'}-\e_{i'}$. We have $\la\lambda+\rho,\alpha^\vee\ra=\breve{\lambda}_j+j-(\breve{\lambda}_i+i)$ and $s_{\alpha,l}\cdot\lambda=\lambda-a\alpha$. So $s_{\alpha,l}(\lambda+\rho)\breve{{}_i}=\breve{\lambda}_i+i+a$ and $s_{\alpha,l}(\lambda+\rho)\breve{{}_j}=\breve{\lambda}_j+j-a$.

We have $a\le\la\lambda+\rho,\alpha^\vee\ra-p\le m+r-(\breve{\lambda}_i+i)-p$. So $\breve{\lambda}_i+i+a\le m+r-p<m-r$, since $2r<p$. It follows that $i<m-r$ and $\breve{\lambda}_i=0$. Now $i<i+a<\breve{\lambda}_j+j$ and $i+a<m-r$. So $s_{\alpha,l}\cdot\lambda+\rho=s_{\alpha,l}(\lambda+\rho)$ contains a repeat and $\schi(s_{\alpha,l}\cdot\lambda)=0$.
\end{proof}

\begin{lemgl}\label{lem.2e_i and lem.e_i+e_j}
Let $\Phi_1$ be the set of roots $\e_i+\e_j$, $1\le i<j\le m$, $j>r$ and let $\Phi_2$ be the set of roots $2\e_i$, $1\le i\le m$. Furthermore, let $S_1$ be the set of pairs $(\alpha,l)$ such that $\alpha\in\Phi_1$, $l$ an integer $\ge1$, $\la\lambda+\rho,\alpha^\vee\ra-lp>0$ and $\schi(s_{\alpha,l}\cdot\lambda)\ne0$, and let $S_2$ be the corresponding set for $\Phi_2$. Then there exists a map $\varphi:S_1\to\Phi_2$ such that:
\begin{enumerate}[{\rm(i)}]
\item $(\alpha,l)\mapsto(\varphi(\alpha,l),l)$ is a bijection from $S_1$ onto $S_2$.
\item $\schi(s_{\alpha,l}\cdot\lambda)=-\schi(s_{\varphi(\alpha,l),l}\cdot\lambda)$.
\end{enumerate}
\end{lemgl}

\begin{proof}
Let $(\alpha,l)\in S_1$. Write $\alpha=\alpha^\vee=\e_{j'}+\e_{i'}$, $1\le i<j\le m$, $i\le m-r$. Put $a=\la\lambda+\rho,\alpha^\vee\ra-lp$. We have $\breve{\lambda}_i=0$, $\la\lambda+\rho,\alpha^\vee\ra=\breve{\lambda}_j+j+i$ and $s_{\alpha,l}\cdot\lambda=\lambda-a\alpha$. So $s_{\alpha,l}(\lambda+\rho)\breve{{}_i}=i-a$ and $s_{\alpha,l}(\lambda+\rho)\breve{{}_j}=\breve{\lambda}_j+j-a$.

We have $i-a<i<\breve{\lambda}_j+j$. So, if $i-a\ge0$, then $s_{\alpha,l}(\lambda+\rho)$ contains a repeat or a zero. Therefore $a-i>0$. We have $a\le\la\lambda+\rho,\alpha^\vee\ra-p\le m+r+i-p<m-r+i$, since $2r<p$. So $0<a-i\le m-r$. Now $a-i<\breve{\lambda}_j+j$. So if $a-i\ne i$, then $s_{\alpha,l}(\lambda+\rho)$ contains a repeat up to sign. Therefore $a=2i$ and $s_{\alpha,l}(\lambda+\rho)\breve{{}_i}=-i$. Now put $\varphi(l,\alpha)=2\e_{j'}$. Note that $(2\e_{j'})^\vee=\e_{j'}$. So $\la\lambda+\rho,\varphi(l,\alpha)^\vee\ra=\breve{\lambda}_j+j=i+lp$. Furthermore, $s_{\alpha,l}(\lambda+\rho)$ is obtained from $s_{\varphi(l,\alpha),l}(\lambda+\rho)$ by changing the sign of the $i^{\rm th}$ coordinate. This proves (ii) and that $(\alpha,l)\mapsto(\varphi(\alpha,l),l)$ is an injection from $S_1$ to $S_2$.

%Remark: One can actually show that $j,\breve{\lambda}_j+j-2i>m-r$: Assume $j\le m-r$. Then $\breve{\lambda}_j=0$. Clearly $j-a=j-2i\ne0$. If $j-2i>0$, then we have $1\le j-2i<j$ and therefore a repeat ($i$ only changes sign). If $j-2i<0$, then $1\le 2i-j<i$ and we have again a repeat. Now one easily deduces that $\breve{\lambda}_j+j-2i>m-r$.

Now let $(\beta,l)\in S_2$. Write $\beta=2\e_{j'}$, $1\le j\le m$. Define $i$ by the equation $\breve{\lambda}_j+j=i+lp$. Clearly $i>0$. Furthermore, $i\le\breve{\lambda}_j+j-p<j$, since $r<p$ and $i\le\breve{\lambda}_j+j-p<m-r$, since $2r<p$. Put $\alpha=\alpha^\vee=\e_{j'}+\e_{i'}$. From the previous computations it now follows that $(\alpha,l)\in S_1$ and it is clear that $\varphi(\alpha,l)=\beta$. This proves (i).
\end{proof}

\begin{lemgl}\label{lem.e_i+e_j}
Assume $\alpha=\e_i+\e_j$, $1\le i<j\le r$, $\la\lambda+\rho,\alpha^\vee\ra=a+lp$, $a,l>0$ and $\schi(s_{\alpha,l}\cdot\lambda)\ne0$. Then $l=u$. Furthermore, the entries of $s_{\alpha,u}(\lambda+\rho)$ are distinct and strictly positive.
\end{lemgl}

\begin{proof}
First we redefine $i$ and $j$ by replacing $(i,j)$ by $(j',i')$. So $m-r<i<j\le m$ and $\alpha=\alpha^\vee=\e_{j'}+\e_{i'}$. Note that the first $m-r$ entries of $s_{\alpha,l}(\lambda+\rho)\breve{\ }$ form the interval $\{1,2\ldots,m-r\}$. We have $0<a<\la\lambda+\rho,\alpha^\vee\ra=\breve{\lambda}_j+j+\breve{\lambda}_i+i=a+lp$. Clearly $\breve{\lambda}_i+i-a=s_{\alpha,l}(\lambda+\rho)\breve{{}_i}\ne0$.
We have $a\le\breve{\lambda}_j+j+\breve{\lambda}_i+i-p\le m+r+\breve{\lambda}_i+i-p<m-r+\breve{\lambda}_i+i$, since $2r<p$. So if $\breve{\lambda}_i+i-a<0$, then $0<a-\breve{\lambda}_i-i<m-r$ and $s_{\alpha,l}(\lambda+\rho)\breve{\ }$ would contain a repeat up to sign. So we have $\breve{\lambda}_i+i-a>0$. Since $\breve{\lambda}_j+j-a>\breve{\lambda}_i+i-a$ this shows that all entries of $s_{\alpha,l}(\lambda+\rho)$ are (distinct and) strictly positive. If $\breve{\lambda}_i+i-a\le m-r$, then $s_{\alpha,l}(\lambda+\rho)\breve{\ }$ would contain a repeat. So we have $\breve{\lambda}_i+i-a>m-r$. It follows that $a<\breve{\lambda}_i+i-m+r< m+r-(m-r)=2r$. So $0<\la\lambda+\rho,\alpha^\vee\ra-lp=a<2r<p/2$. On the other hand $-\delta-2r=n-2r-up<\la\lambda+\rho,\alpha^\vee\ra-up<n+2r-up=-\delta+2r$. Since $|\delta|+2r<p/2$, this implies that $u=l$.
\end{proof}

We can now finish the proof of \eqref{eq.sumformula}. By Lemmas~\ref{lem.e_i-e_j} and \ref{lem.2e_i and lem.e_i+e_j} we can restrict the sum on the right in \eqref{eq.jantzensum} to positive roots $\alpha=\e_i+\e_j$, $1\le i<j\le r$ with $\la\lambda+\rho,\alpha^\vee\ra-lp>0$. Assume now that for such a root $\alpha$ we have $\schi(s_{\alpha,l}\cdot\lambda)\ne0$. Then we have $l=u$, by Lemma~\ref{lem.e_i+e_j}. Now $\hat\rho_i=\rho_i-(u/2)p$, so $\la\lambda+\hat\rho,\alpha^\vee\ra=\lambda_i+\rho_i+\lambda_j+\rho_j-up= \la\lambda+\rho,\alpha^\vee\ra-lp>0$. So $s_{\alpha,u}\cdot\lambda=s_\alpha\star\lambda$.
On the other hand, if $\la\lambda+\hat\rho,\alpha^\vee\ra>0$ and $\schi(s_\alpha\star\lambda)\ne0$, then $(\alpha,u)$ gives the same nonzero summand in the sum on the right of \eqref{eq.jantzensum}.
\end{proof}

\begin{cornn}
Assume that $|\delta|+2r<p/2$ and that $\delta\ge2r-1$. Then, under the above assumptions, $\Delta_0(\lambda)$ is irreducible for any partition $\lambda$ with $l(\lambda),l(\lambda')\le r$.
\end{cornn}

\begin{proof}
Clearly we may assume that $u\ne0$. The assertion now follows immediately from \eqref{eq.sumformula} and the fact that $\la\lambda+\hat\rho,\alpha^\vee\ra\le-\delta+2r-1$.
\end{proof}

\begin{remnn}
Note that for the proofs of Lemmas~\ref{lem.e_i-e_j} and \ref{lem.2e_i and lem.e_i+e_j} we only needed that $r<2p$ and no assumptions on $\delta$.
\end{remnn}

We now turn our attention to the blocks of the symplectic Schur algebra. Let $S$ be a finite dimensional algebra. Fix a labeling set $X$ for the isomorphism classes of irreducible $S$-modules. For $\lambda\in X$, let $L(\lambda)$ be the corresponding irreducible $S$-module. On $X$ the {\it block relation} is defined as the smallest equivalence relation $\sim$ such that $\Ext^1_S(L(\lambda),L(\mu))\ne0$ implies $\lambda\sim\mu$. When $\lambda\sim\mu$, we say that $L(\lambda)$ and $L(\mu)$ (or just $\lambda$ and $\mu$) are in the same block. If $S$ is quasi-hereditary and if for all $\lambda$, $\Delta(\lambda)$ and $\nabla(\lambda)$ have the same composition factors, then $\sim$ is the smallest equivalence relation such that $[\Delta(\lambda):L(\mu)]\ne0$ implies $\lambda\sim\mu$. Here we denote for an $S$-module $M$ the multiplicity of $L(\mu)$ as a composition factor of $M$ by $[M:L(\mu)]$.

In the remainder of this section we assume that $m\ge r$. To prove the block result, we need some more notation. For $x,y\in\mb R^r$ we write $x\subseteq y$ if $x_i\le y_i$ for all $i\in\{1,\ldots,r\}$. Furthermore, we define $x\cap y=\min(x_i,y_i)_{i\in\{1,\ldots,r\}}$. Note that $x\cap y$ is (strictly) decreasing if this holds for $x$ and $y$. We have $(x+z)\cap(y+z)=(x\cap y)+z$ for any $z\in\mb R^r$. The next lemma is suggested by the proof of \cite[Thm.~4.2]{CdVM1}: there it is shown that on $\Lambda_0^+(m,r)$ conjugacy under the star action is equivalent to \lq\lq balancedness" which clearly has the property stated in (ii).

\begin{prop}\label{prop.Wmin}%The Weyl group/minimum result
Let $\lambda,\mu\in\mb R^r$.
\begin{enumerate}[{\rm (i)}]
\item Assume $\lambda$ and $\mu$ are strictly decreasing. If they are conjugate under the action of $W(C_r)$, then the same holds for $\lambda$ and $\lambda\cap\mu$.
\item Assume $\lambda$ and $\mu$ are decreasing. If they are conjugate under the star action of $W(C_r)$, then the same holds for $\lambda$ and $\lambda\cap\mu$.
\item Assertions (i) and (ii) also hold with $W(C_r)$ replaced by $W(D_r)$.
\end{enumerate}
\end{prop}

\begin{proof}
Clearly (ii) follows from (i) and the second assertion of (iii) follows from the first.

Assume (i) holds and that $\lambda$ and $\mu$ are strictly decreasing. We will show that then the first assertion of (iii) holds. If $\lambda$ contains a zero, then our result follows immediately from (i), so assume that this is not the case. For $x\in\mb R^r$ denote the number of entries $<0$ by $N_{-}(x)$. Then we have for $x,y\in(\mb R\sm\{0\})^r$ that $N_{-}(x\cap y)=\max\{N_{-}(x),N_{-}(y)\}$. Furthermore we have that $x$ and $y$ are conjugate under the action of $W(D_r)$ if and only if they are conjugate under the action of $W(C_r)$ and $N_{-}(x)\equiv N_{-}(y)\ ({\rm mod}\ 2)$. The result now follows.

It remains to prove (i). Assume $\lambda$ and $\mu$ are strictly decreasing and that they are conjugate under the action of $W(C_r)$. By the pigeonhole principle it suffices to show for each real number $a$ the following.\\
1.\ If $a$ and $-a$ occur in $\lambda$ (this includes the case that $a=0$ occurs in $\lambda$), then they occur in $\lambda\cap\mu$.\\
2.\ If $a>0$ and $a$ occurs in $\lambda$, but $-a$ does not occur in $\lambda$, then $a$ or $-a$ occurs in $\lambda\cap\mu$.

Let $a\in\mb R$. If $a$ and $-a$ occur in $\lambda$, then $a$ and $-a$ occur in $\mu$. So for 1 it suffices to show that if $a$ occurs in $\lambda$ and $\mu$, then it occurs in $\lambda\cap\mu$. Let $i$ be the index with $\lambda_i=a$. If $\min\{\lambda_i,\mu_i\}=\lambda_i$, then we are done, so assume $\mu_i<\lambda_i=a$. Then $a$ occurs in $\mu$ before position $i$, so $\mu_j=a$ for some $j<i$. But then $\lambda_j>\lambda_i=a=\mu_j$ and $\min\{\lambda_j,\mu_j\}=a$.

Now assume that $a>0$ and $a$ occurs in $\lambda$, but $-a$ does not occur in $\lambda$. Because of the above we may assume that $a$ does not occur in $\mu$. Then $-a$ occurs in $\mu$. For a contradiction, assume that neither $a$ nor $-a$ occurs in $\lambda\cap\mu$. Let $i_0$ be the index with $\lambda_{i_0}=a$ and let $j_0$ be the index with $\mu_{j_0}=-a$. Then we have $b:=\mu_{i_0}<a$ and $c:=\lambda_{j_0}<-a$. If $j_0\le i_0$, then we would have $-a>c=\lambda_{j_0}\ge\lambda_{i_0}=a$, a contradiction. So $j_0>i_0$. In a picture:

$$\begin{array}{cccccc}
            &      &i_0 &      &j_0   &\\
\lambda\quad&\cdots&a   &\cdots&c     &\cdots\\
            &      &\vee&      &\wedge&      \\
\mu\quad    &\cdots&b   &\cdots&-a    &\cdots
\end{array}$$

\medskip

Now we define recursively a sequence of numbers $\sigma_0,\sigma_1,\sigma_2,\ldots\in\{1,\ldots,r\}$ as follows. $\sigma_0=j_0$. Assume that $k\ge0$ and that $\sigma_k$ is defined. If $\lambda_{\sigma_k}$ occurs in $\mu$, then define $\sigma_{k+1}$ as the index with $\mu_{\sigma_{k+1}}=\lambda_{\sigma_k}$, otherwise we define it as the index with $\mu_{\sigma_{k+1}}=-\lambda_{\sigma_k}$.
To reach a contradiction it is clearly sufficient to show that:\\
(1)\ $\sigma_i<i_0$ or $\sigma_i>j_0$ for all $i>0$ and\\
(2)\ $\sigma_i\ne\sigma_j$ for $i\ne j$.\\
We do this by induction. Let $k\ge0$ and assume that $\sigma_0,\ldots,\sigma_k$ are distinct and satisfy (1). If $\sigma_k\ge j_0$, then $\lambda_{\sigma_k}\le\lambda_{j_0}=c<-a=\mu_{j_0}$ and $-\lambda_{\sigma_k}>a>b=\mu_{i_0}$. If $\sigma_k<i_0$, then $\lambda_{\sigma_k}>\lambda_{i_0}=a>b=\mu_{i_0}$ and $-\lambda_{\sigma_k}<-a=\mu_{j_0}$. So, by the definition of $\sigma_{k+1}$, either $\sigma_{k+1}<i_0$ or $\sigma_{k+1}>j_0$.

Finally we have to show that $\sigma_{k+1}\ne\sigma_j$ for $j\le k$. Assume, again for a contradiction, that $\sigma_{k+1}=\sigma_i$ for some $i\le k$. Because of what we just proved we have $i>0$. We have $\pm\lambda_{\sigma_k}=\mu_{\sigma_{k+1}}=\mu_{\sigma_i}=\pm\lambda_{\sigma_{i-1}}$. Because of the induction hypothesis we have $\sigma_k\ne\sigma_{i-1}$, so $\lambda_{\sigma_k}=-\lambda_{\sigma_{i-1}}$. Since $\mu$ is a signed permutation of $\lambda$, this means that $\lambda_{\sigma_k}$ and $\lambda_{\sigma_{i-1}}$ occur in $\mu$. But then, by the definition of $\mu$, $\mu_{\sigma_{k+1}}=\lambda_{\sigma_k}$ and $\mu_{\sigma_i}=\lambda_{\sigma_{i-1}}$. So $\lambda_{\sigma_k}=\lambda_{\sigma_{i-1}}$, a contradiction.
\end{proof}

\begin{remnn}
We haven't checked whether (i) also holds for not necessarily strictly decreasing $\lambda$ and $\mu$. We certainly don't need this more general result.
\end{remnn}
From now on we assume that $|\delta|+2r<p/2$. For $x\in\mb R^r$, we denote by ${\rm sort}(x)$, the $r$-tuple which is obtained by sorting $x$ in descending order.
\begin{lem}\label{lem.inclusion}
Let $\lambda$ be a partition with $l(\lambda),l(\lambda')\le r$ and let $\alpha=\e_i+\e_j$, $1\le i<j\le r$, be a positive root with $\la\lambda+\hat\rho,\alpha^\vee\ra>0$ and $\schi(s_\alpha\star\lambda)\ne0$. Put $\mu={\rm sort}(s_\alpha(\lambda+\hat\rho))-\hat\rho$. Then all entries of $s_\alpha(\lambda+\hat\rho)$ are distinct and $\mu$ is a partition with $\mu\subsetneqq\lambda$. Let $w$ be the permutation such that $w(s_\alpha(\lambda+\hat\rho))$ is (strictly) decreasing. Then $\mu=w\star s_\alpha\star\lambda$ and $\schi(s_\alpha\star\lambda)={\rm sgn}(w)\schi(\mu)$.
\end{lem}

\begin{proof}
We have $s_\alpha\star\lambda=s_{\alpha,up}\cdot\lambda$. Since $\la\lambda+\hat\rho,\alpha^\vee\ra>0$, we have that $s_\alpha(\lambda+\hat\rho)\subseteq\lambda+\hat\rho$ and the same holds for the $r$-tuple obtained from $s_\alpha(\lambda+\hat\rho)$ by sorting it in descending order. So $\mu\subseteq\lambda$. We have $|\mu|=|s_{\alpha}\star\lambda|<|\lambda|$, so $\mu\ne\lambda$. Recall that the final $m-r$ entries of $s_{\alpha,lp}(\lambda+\rho)$ form the reversed interval $\{m-r,\ldots,2,1\}$. Furthermore, the entries of $s_{\alpha,lp}(\lambda+\rho)=s_\alpha(\lambda+\hat\rho)+\rho-\hat\rho$ are distinct and strictly positive by Lemma~\ref{lem.e_i+e_j} in the proof of Theorem~\ref{thm.sumformula}.
So sorting them in descending order only involves the first $r$ entries. The result now follows from \cite[II.5.9(1)]{Jan} and the fact that $\rho-\hat\rho$ has the same value $\frac{u}{2}p$ on the first $r$ entries.
\end{proof}

We can now deduce from Theorem~\ref{thm.sumformula} a linkage principle for the symplectic Schur algebra.
\begin{lem}\label{lem.linkage}
Let $\lambda$ and $\mu$ be partitions with $l(\lambda),l(\lambda'),l(\mu),l(\mu')\le r$.
\begin{enumerate}[{\rm(i)}]
\item If $[\Delta_0(\lambda):L_0(\mu)]\ne0$, then $\mu\subseteq\lambda$ and $\lambda$ and $\mu$ are conjugate under the star action of $W(D_r)$.
\item If $\lambda,\mu\in\Lambda^+_0(m,r)$ are in the same block of $S_0(n,r)$, then they are conjugate under the star action of $W(D_r)$.
\end{enumerate}
\end{lem}

\begin{proof}
Assertion (ii) follows immediately from (i). We will show (i) by induction on $\lambda$ with respect to the ordering $\subseteq$. If $\mu=\lambda$, then there is nothing to prove. So assume that $\mu\ne\lambda$. Then $\Delta_0(\lambda)$ is not irreducible and we must have $u\ne0$. Furthermore, $L_0(\mu)$ is a composition factor of $\bigoplus_{i>0}\Delta_0(\lambda)^i$. Since the ${\rm ch}\,L_0(\nu)$'s form a basis of $(\mb ZX_0)^{W(C_m)}$, we must have that ${\rm ch}\,L_0(\mu)$ occurs in some $\schi(s_\alpha\star\lambda)$ occurring on the right in \eqref{eq.sumformula}. By Lemma~\ref{lem.inclusion} we have that ${\rm ch}\,L_0(\mu)$ occurs in $\schi(\nu)$ for some partition $\nu$ with $\nu\subsetneqq\lambda$ and $\nu\in W(D_r)\star\lambda$. We can now finish by applying the induction hypothesis.
\end{proof}

The next lemma shows that there are no repetitions or cancelations in the sum on the right in \eqref{eq.sumformula}.

\begin{lem}\label{lem.nocancellations}
Let $\lambda$ be a partition with $l(\lambda),l(\lambda')\le r$ and let $\alpha=\e_i+\e_j$ and $\beta=\e_k+\e_l$, $1\le i<j\le r$ and $1\le k<l\le r$, be positive roots with $\la\lambda+\hat\rho,\alpha^\vee\ra>0$, $\la\lambda+\hat\rho,\beta^\vee\ra>0$, $\schi(s_\alpha\star\lambda)\ne0$ and $\schi(s_\beta\star\lambda)\ne0$. If $\alpha\ne\beta$, then $\schi(s_\alpha\star\lambda)\ne\pm\schi(s_\beta\star\lambda)$.
\end{lem}

\begin{proof}
Put $\mu_\alpha={\rm sort}(s_\alpha(\lambda+\hat\rho))-\hat\rho$ and $\mu_\beta={\rm sort}(s_\beta(\lambda+\hat\rho))-\hat\rho$. First assume that $i\ne k$, say $i<k$. The we have that the $i^{\rm th}$ value of ${\rm sort}(s_\beta(\lambda+\hat\rho))$ is $(\lambda+\hat\rho)_i$, but the $i^{\rm th}$ value of ${\rm sort}(s_\alpha(\lambda+\hat\rho))$ is strictly less than this value. So $\mu_\alpha\ne\mu_\beta$.

Now assume that $i=k$ and that $j\ne l$. We have $s_\gamma\star x=x-\la x+\hat\rho,\gamma^\vee\ra\gamma$.
So for the coordinate sum $|x|$ of $x$ we have $|s_\gamma\star x|=|x|-2\la x+\hat\rho,\gamma^\vee\ra$ if $\gamma$ is a short root. Now $\la\lambda+\hat\rho,\alpha^\vee\ra\ne\la\lambda+\hat\rho,\beta^\vee\ra$ and we always have $|{\rm sort}(x)|=|x|$, so $\mu_\alpha$ and $\mu_\beta$ are partitions of different numbers.
\end{proof}

\begin{cornn}
Let $\lambda$ be a partition with $l(\lambda),l(\lambda')\le r$ and let $\Lambda_\lambda$ be the set of partitions $\nu$ such that $\nu={\rm sort}(s_\alpha(\lambda+\hat\rho))-\hat\rho$ for some positive root $\alpha=\e_i+\e_j$ with $\la\lambda+\hat\rho,\alpha^\vee\ra>0$ and $\schi(s_\alpha\star\lambda)\ne0$. Assume that $\Lambda_\lambda\ne\emptyset$ and let $\mu$ be an $\subseteq$-maximal element of $\Lambda_\lambda$. Then $[\Delta_0(\lambda):L_0(\mu)]\ne0$.
\end{cornn}

\begin{proof}
If ${\rm ch}(L_0(\mu))$ occurs in $\schi(\nu)$ for some $\nu\in\Lambda_\lambda$, then $\mu\subseteq\nu$ by Lemma~\ref{lem.linkage} and $\nu=\mu$ by the maximality of $\mu$. So in the sum in \eqref{eq.sumformula} only $\pm\schi(\mu)$ contains ${\rm ch}(L_0(\mu))$ and $\schi(\mu)$ must appear with positive coefficient.
\end{proof}

\begin{lem}\label{lem.chi}
Let $\lambda$ be a partition with $l(\lambda),l(\lambda')\le r$ and let $\alpha=\e_i+\e_j$, $1\le i<j\le r$, be a positive root with $\la\lambda+\hat\rho,\alpha^\vee\ra>0$. Then $\schi(s_\alpha\star\lambda)\ne0$ if and only if $(\lambda+\hat\rho)_i,(\lambda+\hat\rho)_j<\frac{\delta}{2}+r$ and $(\lambda+\hat\rho)_k\ne-(\lambda+\hat\rho)_i$ and $(\lambda+\hat\rho)_k\ne-(\lambda+\hat\rho)_j$ for all $k\in\{1,\ldots,r\}\sm\{i,j\}$.
\end{lem}

\begin{proof}
We have $s_\alpha\star\lambda=s_{\alpha,up}\cdot\lambda$. By Lemma~\ref{lem.e_i+e_j} in the proof of Theorem~\ref{thm.sumformula}, $\schi(s_\alpha\star\lambda)\ne0$ if and only if $s_{\alpha,up}(\lambda+\rho)$ has no repetitions and all its entries are strictly positive. The latter is the case if and only if $s_\alpha(\lambda+\hat\rho)$ has no repetitions and all its entries are $>-\frac{\delta}{2}-r$. But the only entries of $s_\alpha(\lambda+\hat\rho)$ that could be $\le-\frac{\delta}{2}-r$ are the $j^{\rm th}$ entry $-(\lambda+\hat\rho)_i$ and the $i^{\rm th}$ entry $-(\lambda+\hat\rho)_j$.
\end{proof}

\begin{lem}\label{lem.notirreducible}
Let $\lambda$ be a partition with $l(\lambda),l(\lambda')\le r$. Then $\Delta_0(\lambda)$ is not irreducible if and only if there exists a partition $\mu\subsetneqq\lambda$ which is conjugate to $\lambda$ under the star action of $W(D_r)$.
\end{lem}

\begin{proof}
First we note that $\Delta_0(\lambda)$ is not irreducible if and only if the sum on the right in \eqref{eq.sumformula} is nonzero. If $\Delta_0(\lambda)$ is not irreducible, then a $\mu$ as stated must exist by Lemma~\ref{lem.inclusion}. Now assume that such a $\mu$ exists. By Lemma~\ref{lem.nocancellations} it suffices to show that there exists a positive root $\alpha=\e_i+\e_j$, $1\le i<j\le r$, with $\la\lambda+\hat\rho,\alpha^\vee\ra>0$ and $\schi(s_\alpha\star\lambda)\ne0$.

Put $\tilde\lambda=\lambda+\hat\rho$ and $\tilde\mu=\mu+\hat\rho$. Let $w\in W(D_r)$ such that $\tilde\mu=w(\tilde\lambda)$. Write $w=w_1w_2$, where $w_1$ is a permutation and $w_2$ changes an even number of signs. Note that $w_1$ is the unique permutation that sorts $w_2(\tilde\lambda)$ into descending order. Let $I\subseteq\{1,\ldots,r\}$ be the set of indices whose signs are changed by $w_2$. For $i\in I$ and $k\in\{1,\ldots,r\}\sm I$ we have $\tilde\lambda_i\ne-\tilde\lambda_k$, since all entries of $\tilde\mu$ are distinct. If for some $i,j\in I$ with $i\ne j$ we have $\tilde\lambda_i=-\tilde\lambda_j$ then the sign changes on the $i^{\rm th}$ and $j^{\rm th}$ position cancel each other for $\tilde\lambda$. So, after modifying $w_1$, $w_2$ and $I$, we may assume that this does not happen and then we have $\tilde\lambda_i\ne-\tilde\lambda_k$ for any $i\in I$ and $k\in\{1,\ldots,r\}\sm \{i\}$. Note that $I\ne\emptyset$, since $\tilde\mu\ne\tilde\lambda$ and $\tilde\mu$ and $\tilde\lambda$ are strictly descending. Furthermore $|I|$ is even, since we removed an even number of indices from the original set $I$.

We have $2\sum_{i\in I}\tilde\lambda_i=|\tilde\lambda-\tilde\mu|=|\lambda-\mu|>0$. But then there must exist $i,j\in I$, $i\ne j$, such that $\tilde\lambda_i+\tilde\lambda_j>0$. Put $\alpha=\e_i+\e_j$. Then $\la\lambda+\hat\rho,\alpha^\vee\ra=\tilde\lambda_i+\tilde\lambda_j>0$. Since all entries of $\hat\rho$ are $>-\frac{\delta}{2}-r$, the same holds for the entries of $\tilde\mu$. In particular this holds for $-\tilde\lambda_i$ and $-\tilde\lambda_j$. Since $i,j\in I$, we have that for all $k\in\{1,\ldots,r\}\sm\{i,j\}$, $\tilde\lambda_k\ne-\tilde\lambda_i$ and $\tilde\lambda_k\ne-\tilde\lambda_j$. So, by Lemma~\ref{lem.chi}, $\schi(s_\alpha\star\lambda)\ne0$.
\end{proof}

Our proof of the block result for $S_0(n,r)$ is very similar to that for the Brauer algebra in \cite[Cor.~6.7]{CdVM1}.

\begin{thm}\label{thm.sympschurblocks}
Assume that $m\ge r$ and that $|\delta|+2r<p/2$. Let $\lambda,\mu\in\Lambda_0^+(r,r)$. Then $\lambda$ and $\mu$ are in the same block of $S_0(n,r)$ if and only if they are conjugate under the star action of $W(D_r)$.
\end{thm}

\begin{proof}
By Lemma~\ref{lem.linkage}(ii) we only have to show that $\lambda$ and $\mu$ are in the same block of $S_0(n,r)$ if they are conjugate under the star action of $W(D_r)$. By Proposition~\ref{prop.Wmin}(iii) every linkage class under the star action contains a unique $\subseteq$-minimal element. So it suffices to show that if $\lambda$ is not $\subseteq$-minimal in its linkage class, then it is not $\subseteq$-minimal in its block. This follows immediately from Lemma~\ref{lem.notirreducible} and Lemma~\ref{lem.linkage}(i) (or the corollary to Lemma~\ref{lem.nocancellations}).
\end{proof}

\begin{comment}
\begin{remsnn}
1.\ Note that ${\rm sort}(s_\alpha(\lambda+\hat\rho))-\hat\rho$ is a partition if and only if $\lambda+\hat\rho$ has no repetitions and $\min(s_\alpha(\lambda+\hat\rho))\ge\hat\rho_r$. This is precisely the necessary and sufficient condition from Lemma~\ref{lem.chi} for $\schi(s_\alpha\star\lambda)$ to be nonzero.\\
2.\ Note that if $\mu$ and $\lambda$ are conjugate under the dot action of $W(D_r)$, $|\mu|\le|\lambda|$ and $\mu\ne\lambda$, then $\lambda\cap\mu$ is properly contained in $\lambda$.\\
3.\ One can show that if $\mu$ and $\lambda$ are conjugate under the dot action of $W(D_r)$, $\mu\le\lambda$ and $\mu\ne\lambda$, then $|\mu|<|\lambda|$.\\
Here I sketch the argument. Assume for a contradiction that $|\mu|=|\lambda|$.
c and d in the picture below occur on the first position where $\lambda$ and $\mu$
differ and a and b  occur on the last position where $\lambda$ and $\mu$ differ.
Then d and b have to occur in \lambda+\hat\rho between c and a and
-a and -c have to occur in \lambda+\hat\rho between d and b
\lambda+\hat\rho:  ...c..d...b...a...
                      V          ^
\lambda+\hat\rho:  ...d..-a..-c..b...
This would give us c-a>d-b and d-b>-a--c=c-a, a contradiction.
4.\ An example of $\lambda$ and $\mu$ which are conjugate under the dot action of $W(D_r)$ and satisfy $\mu\le\lambda$, but not $\mu\subseteq\lambda$: Take $\delta=5$, $\lambda=(7,4,3,3)$ and $\mu=(5,5,4,1)$. Then $\lambda-\mu=2\e_1-\e_2-\e_3+2\e_4\ge0$ and clearly $\mu\nsubseteq\lambda$.
\end{remsnn}
\end{comment}

\section{The orthogonal Schur algebra and Schur functor}\label{s.orthschur}

Throughout this section we assume that ${\rm char}\,k\ne2$. Furthermore, $n$ is an integer $\ge2$ and we put $m=\lfloor n/2\rfloor$. Let $i\mapsto i'$ be the involution of $\{1,\ldots,n\}$ defined by $i':=n+1-i$ and define the $n\times n$-matrix $J$ with coefficients in $k$ by $J_{ij}=\delta_{ij'}$. So

%\begin{comment}
$$J=
\begin{bmatrix}
0&&1\\
&\rot{75}{$\ddots$}&\\
1&&0\\
\end{bmatrix}.
$$
%\end{comment}
Let $E=k^n$ be the space of column vectors of length $n$ with standard basis $e_1,\ldots,e_n$.
On $E$ we define the nondegenerate symmetric bilinear form $(\ ,\ )$ by
$$(u,v):=u^TJv=\sum_{i=1}^nu_iv_{i'}\ .$$
Then $(e_i,e_j)=J_{ij}$. The {\it orthogonal group} $\Ort_n=\Ort_n(k)$ is defined as the group of $n\times n$-matrices over $k$ that satisfy $A^TJA=J$, that is, the invertible matrices for which the corresponding automorphism of $E$ preserves the form $(\ ,\ )$. The {\it special orthogonal group} $\SO_n=\SO_n(k)$ consists of the matrices in $\Ort_n$ that have determinant $1$. The vector space $E$ is the natural module for $\GL_n$ and for $\Ort_n$ and $\SO_n$. We denote the maximal torus of $\SO_n$ that consists of the diagonal matrices by $T_1$. Then $T_1$ consists of the diagonal matrices $t$ with $t_it_{i'}=1$ for all $i\in\{1,\ldots,m\}$ and with $t_{m+1}=1$ in case $n$ is odd. The character group of $T_1$ can be identified with $\mb Z^m$. The root system of $\SO_n$ with respect to $T_1$ is of type $B_m$ if $n$ is odd and of type $D_m$ if $n$ is even. We choose the set of roots of $T_1$ in the Lie algebra of the Borel subgroup of upper triangular matrices in $\SO_{2n}$ as the system of positive roots. A weight $\lambda\in\mb Z^m$ is dominant if and only if $\lambda_1\ge\lambda_2\cdots\ge\lambda_{m-1}\ge\lambda_m\ge0$ in case $n$ is odd and it is dominant if and only if $\lambda_1\ge\lambda_2\cdots\ge\lambda_{m-1}\ge|\lambda_m|$ in case $n$ is even. We define the $\e_i$ as in Subsection~\ref{ss.sympschur}. We note that the group of weights of $\SO_n$ is a proper subgroup of the weight lattice of the root system, since $\SO_n$ is not simply connected. For $\lambda$ a dominant weight of $\SO_n$ we denote the corresponding induced module by $\nabla_1(\lambda)$. The Schur algebra, Schur functor and inverse Schur functor are as defined in Sections~\ref{s.prelim} and \ref{s.sympschur}. The definitions and results for the Schur algebra given there are of course also valid for odd $n$. Let $r$ be an integer $\ge0$. We define the {\it orthogonal Schur algebra} $S_1(n,r)$ to be the enveloping algebra of $\Ort_n$ in $\End_k(E^{\otimes r})$.

Let $B_r=B_r(n)$ be the Brauer algebra. There is a natural homomorphism $B_r\to\End_{\Ort_n}(E^{\otimes r})$. As in the symplectic case, one shows using classical invariant theory that this homomorphism is surjective and that it is injective if $n\ge 2r$. This is completely analogous to the symplectic case, see \cite[Sect.~3]{T}. In \cite{Cliff} a bideterminant basis is given for $k[\Ort_n]$. Cliff has informed us that his arguments also show that $k[\Mat_n]$ modulo the relations (60) in \cite{Doty} is spanned by $\Ort_n$-standard bideterminants multiplied by a power of the coefficient of dilation. From this one deduces that these relations generate the vanishing ideal of the orthogonal monoid. Now it follows that $\End_{B_r}(E^{\otimes r})=S_1(n,r)$ by \cite[Rem.~3.3]{T}. See \cite{DotyHu} for another approach to the double centralizer theorem for the orthogonal group.

By \cite[Prop.~3.3(iii)]{Brun} every $\GL_n$-module with a good $\GL_n$-filtration also has a good $\SO_n$-filtration.
%Actually Brundan's proof doesn't look quite right. The following remarks should help to give a correct proof.
%First assume n is odd. Then the \nabla(\varpi_i), i=1,...,m-1 are the exterior powers \bigwedge^iE, i=1,...,m-1, which are irreducible; see e.g. \cite{Wong} and \cite[4.9]{anjan}. The m-th exterior power is well known to be irreducible in characteristic 0. In characteristic p it should also be irreducible or at least isomorphic to \nabla(\varpi_m) (here \varpi_m denotes the all one vector of length m).
%Now assume n is even. Then the \nabla(\varpi_i), i=1,...,m-2 are the exterior powers \bigwedge^iE, i=1,...,m-2, which are irreducible; see again \cite{Wong} and \cite[4.9]{anjan}. The (m-1)-th exterior power is well known to be irreducible in characteristic 0. In characteristic p it should also be irreducible or at least isomorphic to \nabla(\varpi_{m-1}) (this is the type A \varpi_{m-1}). The m-th exterior power splits into a direct sum of two modules which are irreducible in char o. This can be seen using the Hodge star operator. See p290 (in Sect. 19.2) of Fulton-Harris or the two flag algebra papers by Lancaster and Towber. In characteristic p these modules should be the induced modules corresponding to the final two fundamental weights in type D_m. Maybe they are actually irreducible in char p.
From this one easily deduces that for every partition $\lambda$ of length at most $m$, restriction of functions defines an epimorphism $\nabla(\lambda)\to\nabla_1(\lambda)$ (Note that in case $n$ is even, $\SO_n$ also has other dominant weights). It also follows that a tilting module for $\GL_n$ is a tilting module for $\SO_n$, by restriction. Now let $S^0_1(n,r)$ be the enveloping algebra of $\SO_n$ in $\End_k(E^{\otimes r})$. Clearly $S^0_1(n,r)^*$ is the coefficient space of the $\SO_n$-module $E^{\otimes r}$. Recall the definition of $\Lambda_0^+(m,r)$ from Subsection~\ref{ss.sympschur}. We now consider the cases $n$ is even and $n$ is odd separately.

First assume that $n$ is even and $>2r$. Then $m>r$ and it is not hard to check that the set $\Lambda_0^+(m,r)$ is saturated.
%Let $\lambda$ be a dominant weight of $\SO_n$ and assume $\lambda\le\mu\in\Lambda_0^+(m,r)$. Then $|\lambda|$ and $|\mu|$ have the same parity. Furthermore, taking inner products with $\varpi_m$ and $\varpi_{m-1}$, we get $|\lambda|\le|\mu|\le r<m$ and $(\sum_{i=1}^{m-1}\lambda_i)-\lambda_m\le|\mu|\le r<m$. Here we used that $\mu_m=0$. Since either $\lambda$ or $(\lambda_1,\ldots,\lambda_{m-1},-\lambda_m)$ is a partition, we get that $\lambda_m=0$.
The arguments in \cite[Sect.~8]{Don8} now show that $S^0_1(n,r)^*=O_{\Lambda_0^+(m,r)}(k[\SO_n])$ which means that $S^0_1(n,r)$ is the generalized Schur algebra associated to $\SO_n$ and $\Lambda_0^+(m,r)$. We will now show that $S^0_1(n,r)=S_1(n,r)$. The dimension of $S^0_1(n,r)$ is independent of the field by \cite[(2.2c)]{Don1}. By \cite[Prop.~1(i)]{T} the vector space $S_1(n,r)$ is isomorphic to the dual of the vector space $A_{\rm GO}(n,r,k)$ in \cite[Thm.~8.1]{Cliff}, the dimension of which is independent of the field. So we may now assume that ${\rm char}\,k=0$. Then the algebras $S^0_1(n,r)$ and $S_1(n,r)$ are semisimple, so it suffices to check that their centralizers coincide. By the first fundamental theorem of invariant theory for $\SO_n$, \cite[Sect.~II.9]{Weyl} or \cite[Thm.~5.6(ii)]{DeCProc}, we have that $k[\oplus^{2r}E]^{\SO_n}=k[\oplus^{2r}E]^{\Ort_n}$, since $n>2r$. Taking the multilinear invariants on both sides we obtain $\End_{\SO_n}(E^{\otimes r})\cong(E^{\otimes 2r})^{*,\SO_n}=(E^{\otimes 2r})^{*,\Ort_n}\cong\End_{\Ort_n}(E^{\otimes r})$. The two isomorphisms here were observed by Brauer in \cite{Br}.

Now we treat the case $n$ is odd. In this case the set $\Lambda_0^+(m,r)$ is clearly not saturated, since $(r-1)\e_1\le r\e_1\in\Lambda_0^+(m,r)$ and $(r-1)\e_1\notin\Lambda_0^+(m,r)$. So we will proceed differently. First note that $S^0_1(n,r)=S_1(n,r)$, since $-\id\in\Ort_n\sm\SO_n$.

Let $S$ be the enveloping algebra of $\Ort_n$ in $\End_k(E^{\otimes r}\oplus E^{\otimes (r-1)})$. By \cite[Prop.~1(ii)]{T} we have that $S^*\cong k[\Ort_n]^{\le r}$, the coalgebra of regular functions on $\Ort_n$ of filtration degree $\le r$ (the summands $E^{\otimes s}$, $s<r-1$, occurring in \cite{T} can be omitted). The defining ideal of $\Ort_n$ is homogeneous for the $\mb Z_2$-grading of $k[\Mat_n]$ (the two graded subspaces are the sums of the $\mb Z$-graded subspaces of even and odd degree respectively), so $k[\Ort_n]^{\le r}$ has a direct sum decomposition into two sub coalgebra summands, and therefore $S$ has a direct sum decomposition in two ideal summands: $S=S(0)\oplus S(1)$. We have $S_1(n,r)=S(0)$ if $r$ is even and $S_1(n,r)=S(1)$ if $r$ is odd. The bideterminant basis of $k[\Ort_n]$ in \cite[Cor~6.2]{Cliff} gives a basis for the subspace $k[\Ort_n]^{\le r}$ which is labelled by pairs of $\Ort_n$-standard tableaux of some shape $\lambda$ with $|\lambda|\le r$.
%See also \cite[Thm.~8.1]{Cliff}.
In particular $k[\Ort_n]^{\le r}$, and therefore $S$, has dimension independent of the field.

Now let $S^0$ be the enveloping algebra of $\SO_n$ in $\End_k(E^{\otimes r}\oplus E^{\otimes (r-1)})$ and let $\pi$ be the set of dominant weights $\{\lambda\in\Lambda^+(m)\,|\,|\lambda|\le r\}$. The set $\pi$ is clearly saturated and one checks using the arguments in \cite[Sect.~8]{Don8} that $S^0$ is the generalized Schur algebra associated to $\SO_n$ and $\pi$.

Now assume furthermore that $n>2r$. Then we deduce as in the case $n$ is even that $S^0=S$. This also means that the restriction of functions $k[\Ort_n]^{\le r}\to k[\SO_n]^{\le r}$ is an isomorphism. So $S$ is the generalized Schur algebra associated to $\SO_n$ and $\pi$, and one easily checks that the direct sum decomposition $S=S(0)\oplus S(1)$ corresponds to the partition of $\pi$ into the sets $\Lambda_0^+(m,r)$ and $\pi\sm\Lambda_0^+(m,r)$. So $S_1(n,r)=O_{\Lambda_0^+(m,r)}(k[\SO_n])^*$ is a direct ideal summand of the generalized Schur algebra $S$ and therefore a quasihereditary algebra; $\Lambda_0^+(m,r)$ is the set of labels of its irreducibles. From the arguments in \cite[2.2]{Don1} it is now also clear that $S_1(n,r)$ has dimension $\sum_{\lambda\in\Lambda_0^+(m,r)}\dim(\nabla_1(\lambda))^2$, which is independent of the field $k$.

\smallskip
In the remainder of this section we assume that $n>2r$.
\smallskip

We define the {orthogonal Schur functor} $f_1:{\rm mod}(S_1(n,r))\to{\rm mod}(B_r)$ and the {\it inverse orthogonal Schur functor} $g_1:{\rm mod}(B_r)\to{\rm mod}(S_1(n,r))$ in precisely the same way as in the symplectic case.

The proofs of the orthogonal versions of Lemmas~\ref{lem.dimension}, \ref{lem.diagram} and \ref{lem.surjective} are completely analogous. Once the orthogonal version of (*) in Lemma~\ref{lem.dimension} is proved for all $r$ with $2(r+1)<n$, the rest of the proof is the same as in the symplectic case.
%Assume n=2m is even. Under the assumption n>2r, a partition and its associate cannot both have partition number \le r. In particular, self-associate partitions don't show up in E^{\otimes r}. So the recursive formula is only valid for r with 2(r+1)<n. If you want to prove a version which is also valid for n=2r, you have to work with the full orthogonal group.
For the weight space argument in Lemma~\ref{lem.diagram}(ii) one can use the orthogonal standard tableaux from \cite{KingWelsh}. Now we obtain the orthogonal version of Theorem~\ref{thm.sympschur}.

\begin{thm}\label{thm.orthschur}
The following holds.
\begin{enumerate}[{\rm(i)}]
\item For $\lambda\in\Lambda^+_0(m,r)$ we have
\begin{align*}
f_1(\nabla_1(\lambda))&\cong {\mc S}(\lambda),\\
f_1(S^\lambda E)&\cong {\mc M}(\lambda)\text{\quad and}\\
f_1(\bigwedge{\Nts}^\lambda E)&\cong \widetilde{\mc M}(\lambda)\text{\quad if ${\rm char}\,k=0$ or $>|\lambda|$.}
\end{align*}
\item Let $M$ be an $S(n,t)$-module. If $M$ is a direct sum of direct summands of $E^{\otimes t}$ or if $M$ is injective, then
$$f_1(M)\cong {\rm Ind}^{B_r}_{k\Sym_t}f(M).$$
\end{enumerate}
\end{thm}

For $\lambda\in\Lambda_0^+(m,r)$ we denote the indecomposable tilting module for $\SO_n$ of highest weight $\lambda$ by $T_1(\lambda)$ and for $\lambda$ $p$-regular we denote the projective cover of the irreducible $B_r$-module ${\mc D}(\lambda)$ by ${\mc P}(\lambda)$. The proof of the orthogonal version of Proposition~\ref{prop.sympmult} and its corollary are completely analogous.
%An exception is that we don't have $w_0=-\id$ for the longest element $w_0$ of $W(D_m)$ if $m$ is odd, but we still have that $-w_0(\lambda)=\lambda$ for all $\lambda\in\Lambda_0(m,r)$, since $n>2r$.
\begin{prop}
Let $\lambda\in\Lambda_0^+(m,r)$. Then $T_1(\lambda)$ is a direct summand of the $\SO_n$-module $E^{\otimes r}$ if and only if $\lambda$ is $p$-regular and $\lambda\ne\emptyset$ in case $r$ is even $\ge2$ and $\delta=0$. Now assume that $\lambda$ satisfies these conditions. Then
\begin{enumerate}[{\rm(i)}]
\item $f_1(T_1(\lambda))={\mc P}(\lambda)$.
\item The multiplicity of $T_1(\lambda)$ in $E^{\otimes r}$ is $\dim{\mc D}(\lambda)$.
\item The decomposition number $[{\mc S}(\mu):{\mc D}(\lambda)]$ is equal to the $\Delta$-filtration multiplicity $(T_1(\lambda):\Delta_1(\mu))$ and to the $\nabla$-filtration multiplicity $(T_1(\lambda):\nabla_1(\mu))$.
\end{enumerate}
\end{prop}

For a fixed integer $m'$ and a partition $\lambda$ with $l(\lambda)\le m$ and $l(\lambda')=\lambda_1\le m'$ we define $\lambda^\dagger$ as in Section~\ref{s.sympschur}. In the corollary below we apply the previous proposition in the case that $n=2m$ is even.

\begin{cornn}
Let $\lambda,\mu\in\Lambda_0^+(r,r)$ with $\lambda$ $p$-regular. Assume that $\lambda_1,\mu_1\le m'$. Then we have the equality of decomposition numbers
$$[{\mc S}(\mu):{\mc D}(\lambda)]=[\nabla_1'(\lambda^\dagger):L_1'(\mu^\dagger)]\,,$$
where $\nabla'_1$ and $L'_1$ denote induced and irreducible modules for $\SO_{2m'}$.
\end{cornn}

The orthogonal versions of Remarks~\ref{rems.sympschur}, 1, 2 and 4 and Lemma~\ref{lem.tensor} can be proved in precisely the same way. We have the orthogonal version of \eqref{eq.g_0}
$$f_1(A_1(n,r))=\Hom_{\SO_n}(E^{\otimes r},S_1(n,r)^*)\cong\Hom_{\SO_n}(S_1(n,r),E^{\otimes r})\cong E^{\otimes r}\text{\ and}$$
\begin{equation}\label{eq.g_1}
g_1(E^{\otimes r})=E^{\otimes r}\otimes_{B_r}E^{\otimes r}\cong\End_{B_r}(E^{\otimes r})^*=S_1(n,r)^*\cong A_1(n,r).
\end{equation}

Now we obtain the orthogonal version of Proposition~\ref{prop.sympyoung}.

\begin{prop}\label{prop.orthyoung}\
\begin{enumerate}[{\rm(i)}]
\item If $n=0$ in $k$ and $t=0$, assume $r\ge4$. Then we have
$$g_1({\rm Ind}^{B_r}_{k\Sym_t} V)\cong g(V)$$
as $\SO_n$-modules, for every $k\Sym_t$-module $V$.
\item Let $\lambda\in\Lambda^+_0(m,r)$. If $\lambda=\emptyset$ and $m=0$ in $k$, then assume $r\ge4$. Then $g_1({\mc M}(\lambda))\cong S^\lambda E$ and $g_1(\widetilde{\mc M}(\lambda))\cong \bigwedge{\Nts}^\lambda E$.
\item Let $\lambda\in\Lambda^+_0(m,r)$. The $\SO_n$-module $S^\lambda E$ has a unique indecomposable summand $J(\lambda)$ in which $\nabla_1(\lambda)$ has filtration multiplicity $>0$ and this multiplicity is equal to $1$. Every summand of ${\mc M}(\lambda)$ has a Specht filtration and $f_1(J(\lambda))\cong{\mc Y}(\lambda)$.
\end{enumerate}
\end{prop}
We finally note that the orthogonal versions of Remarks~\ref{rems.g_0} are also valid. Only Remark~\ref{rems.g_0}.3 needs some modifications in the case that $n$ is odd.

\section{Blocks}\label{s.blocks}
\subsection{The blocks of the Brauer algebra and the symplectic and
orthogonal Schur algebras in characteristic $p$}\label{ss.sameblocks}
In this subsection we assume that the field $k$ is of positive characteristic $p$. Furthermore, $n$ is an integer $\ge2$ and we put $m=\lfloor n/2\rfloor$. Let $\delta$ be an integer. Recall from Section~\ref{s.jantzen} (the paragraph before Proposition~\ref{prop.Wmin}) that the block relation is defined on a labeling set for the irreducibles. Since cell modules of a cellular algebra always belong to one block, we can extend the block relation of $B_r(\delta)$ to an equivalence relation on all of $\Lambda_0^+(r,r)$ (not just the $p$-regular partitions) as follows: $\lambda$ and $\mu$ are in the same block if and only if ${\mc S}(\lambda)$ and ${\mc S}(\mu)$ belong to the same block. Note that if $p>r$, we are only extending the block relation if $r$ is even $\ge2$ and $\delta$ is zero in $k$. In this case we add the empty partition.
Let $\lambda$ be a partition of $t$, $t\le r$ with $r-t=2s$ even. We have $k_{\rm sg}\otimes S(\lambda)\cong S(\lambda')^*$. Since every simple $k\Sym_t$-module is self-dual, we have that $V$ and $V^*$ have the same composition factors (with multiplicities) for every finite dimensional $k\Sym_t$-module $V$. So $k_{\rm sg}\otimes S(\lambda)$ and $S(\lambda')$ are in the same $k\Sym_t$-block. Since the functor $Z_s\otimes_{k\Sym_t}-$ is exact we get that $\widetilde{\mc S}(\lambda)$ and ${\mc S}(\lambda')$ are in the same $B_r(\delta)$-block.

To prove our next result, we need the following basic fact about quasihereditary algebras. For lack of reference we include a proof. For the general theory of quasihereditary algebras we refer to the appendix of \cite{Don7}.

\begin{lem}\label{lem.ringeldual}
Let $S$ be a finite dimensional quasihereditary algebra with partially ordered labeling set $(X,\le)$ for the irreducibles. Let $S'$ be the Ringel dual of $S$ with reversed partial order $\le'=\le^{\rm op}$ on the labeling set $X$ and let $\lambda,\mu\in X$. Then $\lambda$ and $\mu$ are in the same $S$-block if and only if they are in the same $S'$-block.
\end{lem}

\begin{proof}
We have $S'=\End_S(T)^{\rm op}$ for a full tilting module $T$ of $S$. Denote the irreducible, indecomposable projective and indecomposable tilting module with label $\lambda$ by $L(\lambda)$, $P(\lambda)$ and $T(\lambda)$ respectively. The analogues for $S'$ are \lq\lq primed". Recall that we have the canonical functor $F=\Hom_S(T,-):{\rm mod}(S)\to{\rm mod}(S')$. Since $(S')'$ is Morita-equivalent to $S$, it suffices to show that $\lambda$ and $\mu$ are in the same $S$-block if they are in the same $S'$-block. The block relation of $S'$ is generated by the relation $[P'(\lambda):L'(\mu)]\ne0$, so it suffices to show that $[P'(\lambda):L'(\mu)]\ne0$ implies that $\lambda$ and $\mu$ are in the same $S$-block. So assume the former. We have $[P'(\lambda):L'(\mu)]=\dim\Hom_{S'}(P'(\mu),P'(\lambda))$. Since $F(T(\nu))=P'(\nu)$ for every $\nu\in X$, we have that $\Hom_S(T(\mu),T(\lambda))\ne0$, by the isomorphism \cite[Prop.~A4.8(i)]{Don7}. This clearly implies that $\lambda$ and $\mu$ are in the same $S$-block.
\end{proof}

\begin{thm}\label{thm.sameblocks}
Let $\lambda,\mu\in\Lambda_0^+(r,r)$.
\begin{enumerate}[{\rm(i)}]
\item Assume $n$ is even and $m\ge r$. Then $\lambda$ and $\mu$ are in the same $S_0(n,r)$-block if and only if $\lambda'$ and $\mu'$ are in the same $B_r(-n)$-block.
\item Assume $p\ne2$ and $n>2r$. Then $\lambda$ and $\mu$ are in the same $S_1(n,r)$-block if and only if they are in the same $B_r(n)$-block.
\end{enumerate}
\end{thm}

\begin{proof}
(i).\ First assume that $\lambda'$ and $\mu'$ are in the same $B_r(-n)$-block. As we have seen, this means that $\widetilde{\mc S}(\lambda)$ and $\widetilde{\mc S}(\mu)$ are in the same $B_r(-n)$-block. Let $T$ be the full tilting module $\bigoplus_{\nu\in\Lambda_0^+(r,r)}\bigwedge{\nts}^\nu E$ and let $S'=\End_{S_0(n,r)}(T)$ be the Ringel dual of $S_0(n,r)$. Let $F:{\rm mod}(S_0(n,r))\to{\rm mod}(S')$ be as above. Denote the projection of $T$ onto $E^{\otimes r}$ by $e$. Then we have $eS'e=B_r(-n)$ and $f_0(M)=eF(M)$ for every $S_0(n,r)$-module $M$. By Theorem~\ref{thm.sympschur}(i) we have $e\Delta'(\nu)=\widetilde{\mc S}(\nu)$ for all $\nu\in X$. Here we have used that $F(\nabla_0(\nu))=\Delta'(\nu)$, the standard module with label $\nu$ of $S'$. By assumption there exists an indecomposable summand (block) $A$ of $B_r(-n)$ such that $A$ is nonzero on $\widetilde{\mc S}(\lambda)$ and $\widetilde{\mc S}(\mu)$. Let $R$ be the indecomposable summand of $S'$ such that $A\subseteq eRe$. Then $R$ is nonzero on $\Delta'(\lambda)$ and $\Delta'(\mu)$. It follows that $\lambda$ and $\mu$ are in the same $S'$-block and therefore also in the same $S_0(n,r)$-block, by Lemma~\ref{lem.ringeldual}.

Now we will show that $\lambda'$ and $\mu'$ are in the same $B_r(-n)$-block if $\lambda$ and $\mu$ are in the same $S_0(n,r)$-block. Since the relation $[\Delta_0(\mu):L_0(\lambda)]\ne0$ generates the block relation of $S_0(n,r)$, we may assume that $[\Delta_0(\mu):L_0(\lambda)]\ne0$. Then $(I_0(\lambda):\nabla_0(\mu))=[\Delta_0(\mu):L_0(\lambda)]\ne0$; see e.g. \cite[Prop. A2.2]{Don7}. Here $I_0(\lambda)\subseteq A_0(n,r)$ denotes the $S_0(n,r)$-injective hull of $\nabla_0(\lambda)$. Applying the symplectic Schur functor we obtain that $\widetilde{\mc S}(\mu)$ occurs in a twisted Specht filtration of $f_0(I_0(\lambda))$. By \eqref{eq.g_0} we have that $f_0(I_0(\lambda))$ is indecomposable; see also Remark~\ref{rems.g_0}. Since, clearly, $\widetilde{\mc S}(\lambda)$ is a submodule of $f_0(I_0(\lambda))$, we get that $\widetilde{\mc S}(\lambda)$ and $\widetilde{\mc S}(\mu)$ are in the same $B_r(-n)$-block. As we have seen, this means that $\lambda'$ and $\mu'$ are in the same $B_r(-n)$-block.

The proof of (ii) is completely analogous. Here we use Theorem~\ref{thm.orthschur}(i) and \eqref{eq.g_1} instead of Theorem~\ref{thm.sympschur}(i) and \eqref{eq.g_0}.
\end{proof}

In the corollaries below the star action of $W(D_r)$ is defined as in Section~\ref{s.jantzen}. The definitions of $\hat\rho$ and the star action given there make sense for any integer $\delta$.

\begin{corgl}\label{cor.brauerblocks}
Let $\delta$ be an integer. Assume that $p\ne2$ and that $|\delta|+2r<p/2$. Let $\lambda,\mu\in\Lambda_0^+(r,r)$. Then $\lambda$ and $\mu$ are in the same block of $B_r(\delta)$ if and only if $\lambda'$ and $\mu'$ are conjugate under the star action of $W(D_r)$.
\end{corgl}

\begin{proof}
Choose and integer $u$ such that $\delta-up=-2m$, where $m\ge r$. Now apply Theorem~\ref{thm.sameblocks}(i) and Theorem~\ref{thm.sympschurblocks}.
\end{proof}

\begin{corgl}
Let $\delta$ be the unique integer with $|\delta|<p/2$ and $n-\delta\in p\mb Z$. Assume that $p\ne2$, that $n>2r$ and that $|\delta|+2r<p/2$. Let $\lambda,\mu\in\Lambda_0^+(r,r)$. Then $\lambda$ and $\mu$ are in the same block of $S_1(n,r)$ if and only if $\lambda'$ and $\mu'$ are conjugate under the star action of $W(D_r)$.
\end{corgl}
\setcounter{corgl}{0}
\begin{proof}
This follows immediately from Theorem~\ref{thm.sameblocks}(ii) and the preceding corollary.
\end{proof}

\begin{rems}
1.\ Assume $p\ne2$. Let $\delta$ be any integer and let $r$ be an integer $\ge 0$. Let $W_p(D_r)$ be the affine Weyl group of type $D_r$. One can define the star action of $W_p(D_r)$ in the same way as for $W(D_r)\subseteq W(C_r)$. Cox, De Visscher and Martin \cite[Cor.~6.3]{CdVM2} obtained the following linkage principle. Let $\lambda,\mu\in\Lambda_0^+(r,r)$. Then $\lambda'$ and $\mu'$ are conjugate under the star action of $W_p(D_r)$ if $\lambda$ and $\mu$ are in the same $B_r(\delta)$-block. From Theorem~\ref{thm.sameblocks}(i) we now deduce that for $n=2m\ge 2r$, $\lambda$ and $\mu$ are conjugate under the star action of $W_p(D_r)$ if they are in the same $S_0(n,r)$-block; here $\delta$ is the unique integer with $|\delta|<p/2$ and $-n-\delta\in p\mb Z$. Similarly, we deduce that for $n>2r$, $\lambda'$ and $\mu'$ are conjugate under the star action of $W_p(D_r)$ if $\lambda$ and $\mu$ are in the same $S_1(n,r)$-block; here $\delta$ is the unique integer with $|\delta|<p/2$ and $n-\delta\in p\mb Z$.

From the linkage principles for the symplectic and orthogonal group we could deduce similar, but weaker linkage principles.\\
2.\ Assume that $p>r$. Then the Schur algebra $S(n,r)$ is semisimple. Now assume further that $n=2m$ is even. Recall that $\bigoplus_{\lambda\in\Lambda_0^+(m,r)}\bigwedge{\nts}^\lambda E$ is a full tilting module for $S_0(n,r)$. The natural epimorphisms $E^{\otimes t}\to \bigwedge{\nts}^\lambda E$ are split, since they are $S(n,r)$-epimorphisms. From Lemma~\ref{lem.directsummand} we now deduce that $E^{\otimes r}$ is a full tilting module if $n\ne0$ in $k$ and that, in general, $E^{\otimes r}\oplus k$ is a full tilting module. In particular, if $n=2m$, $p>r$ and $n\ne0$ in $k$, then the Brauer algebra $B_r(-n)$ is the Ringel dual of the symplectic Schur algebra $S_0(n,r)$. Similar remarks apply to $B_r(n)$ and orthogonal Schur algebra $S_1(n,r)$.
\end{rems}

\begin{comment}
Assume $p\ne2$. Let $n=2m+1$ be an odd integer $\ge3$ and let $r$ be an integer $\ge0$. Assume $n>2r$. From the linkage principle for $\SO_n$ we immediately deduce the following linkage principle for $S_1(n,r)$. Let $\lambda,\mu\in\Lambda_0^+(r,r)$. Then $\lambda$ and $\mu$ are conjugate under the dot action of $W_p(B_m)$ if they are in the same $S_1(n,r)$-block. From this one easily deduces that $\lambda$ and $\mu$ are conjugate under the star action of $W_p(C_r)$ if they are in the same $S_1(n,r)$-block. Now let $\delta$ be any integer. Then we deduce from Theorem~\ref{thm.sameblocks}(ii) and the argument of Corollary~1 that $\lambda$ and $\mu$ are conjugate under the star action of $W_p(C_r)$ if they are in the same $B_r(\delta)$-block.
\end{comment}

\subsection{Generalities on reduction mod $p$}\label{ss.redmodp}
We shall need that, for a fixed integer $\delta$, the blocks of the Brauer algebra over a field of characteristic zero \lq\lq agree"
with the blocks over a field of large prime characteristic.   In this subsection we
recall the general reduction argument. The notation used here is completely independent from that in the rest of the paper. Let $R$ be a Dedekind domain with field of fractions $K$. We fix a finite dimensional $K$-algebra $A$ and an order $\Lambda$ in $A$. Thus $\Lambda$ is an $R$-subalgebra of $A$ which is finitely generated as an $R$-module and the $K$-span of $\Lambda$ is $A$. We assume that $A$ is split, in the sense that $\End_A(V)=K$
for every irreducible $A$-module $V$. We will show that the separation
into blocks of the irreducible modules over $K$ agrees with that of $\Lambda/M\Lambda$, for all but finitely many maximal ideals $M$ of $R$.

By a lattice we mean a $\Lambda$-module that is finitely generated and torsion free as an $R$-module. If $V$ is an $A$-module of finite $K$-dimension we shall say that a $\Lambda$-submodule $L$ of $V$ is a
(full) lattice in $V$ if $L$ is finitely generated over $R$ and the $K$-span of $L$ is $V$.

We write $\Max(R)$ for the set of maximal ideals of $R$. For $M\in \Max(R)$ and $F=R/M$ we have the finite dimensional $F$-algebra $\Lambda_F=F\otimes_R \Lambda$. If $L$ is a lattice over $\Lambda$ then we obtain a finite dimensional $\Lambda_F$-module $L_F=F\otimes_R L$ by base change.

We will use that an exact sequence $0\to L_1\to L_2\to L_3\to0$ of $R$-modules splits if $L_3$ is finitely generated and torsion free. This follows immediately from the fact that a finitely generated torsion free module over a Dedekind ring is projective.

For a finite dimensional algebra $S$ we write $\Grot(S)$ for the Grothendieck group of finitely generated left $S$-modules. We write $[V]$ for the  class in $\Grot(S)$ of a finitely generated
$S$-module $V$.     Recall that if $V$ is  a finite dimensional
$A$-module  and $L$ is a lattice in $V$ then the class $[L_F]$ is independent
of the choice of the lattice $L$. (See, for example, the argument of \S15.1, Th\'eor\`eme 32 of \cite{Serre1}.)
We have the decomposition homomorphism $\Grot(A)\to\Grot(\Lambda_F)$, taking $[V]$ to $[L_F]$.

We now label the set of maximal ideals $M_t$, $t\in T$ (with $M_s\neq M_t$ for $s\neq t$). We set
$F_t=R/M_t$ and $A_t=\Lambda_{F_t}$,    for $t\in T$.
We fix a complete set of pairwise irreducible $A$-modules $V_1,\ldots,V_n$ and choose corresponding lattices  $L_1,\ldots,L_n$ in these modules. Let $d_i=\dim V_i$, $1\leq i\leq n$.

We note the following.
\bigskip

{(1)\quad\sl For all but finitely many $t\in T$, the algebra $\Lambda_{F_t}$ is split and \\
$L_{1,F_t},\ldots,L_{n,F_t}$ is a complete set of irreducible pairwise non-isomorphic\\  $A_t$-modules.  \bigskip}

First suppose  $A$ is semisimple.  Let $e(i)$ be the central idempotent that acts as the identity on $V_i$ and as $0$ on $V_j$, for $j\neq i$. We have the orthogonal decomposition $1=e(1)+\cdots+e(n)$  of $1\in A$ as a sum of centrally primitive idempotents.

For $1\leq k\leq n$, the algebra  $e(k)A$ is a $d_k\times d_k$ matrix algebra. We choose a total matrix basis $e(k)_{ij}$,  $1\leq i,j\leq d_k$, of $e(k)A$, for $1\leq k\leq m$.  For some $0\neq g \in R$, we have $ge(k)_{ij}\in\Lambda$ for all $k,i,j$.  Now $g$ is contained in only finitely many maximal ideals. If $M_t$ is a maximal ideal not containing $g$ and $ \overline{g}=g+M$ then defining elements $f(k)=\overline{g}^{-1}(1\otimes ge(k))$ and
$f(k)_{ij}=\overline{g}^{-1}(1\otimes ge(k)_{ij})$,   of $A_t$, we
have an orthogonal idempotent  decomposition $1=f(1)+\cdots+f(n)$ in $A_t$ and a total matrix basis $f(k)_{ij}$, $1\leq i,j\leq n_k$, of
$A_t f(k)$, for $1\leq k\leq n$.   In particular $A_t f(k)$ is a
matrix algebra, $f(k)$ is centrally primitive, $1\leq k\leq n$, and $A_t$ has $n$ pairwise non-isomorphic simple modules.  Now $f(k)$ acts as the identity on $L_{k,F_t}$ and $L_{k,F_t}$ has $F_t$-dimension $d_k$. Hence $L_{k,F}$ is simple and absolutely irreducible  as a $A_t f(k)$-module and hence as a $A_t$-module.  Thus $L_{1,F},\ldots,L_{n,F}$ is a complete set of pairwise non-isomorphic simple $A_t$-modules and $A_t$ is split, for almost all values of $t$.

We now consider the general case. Let $J$ be the Jacobson radical of $A$ and put $I=J\bigcap \Lambda$. Then $I$ acts annihilates  each $V_i$ and hence each $L_i$. Let $t\in T$.  We identify $I_{F_t}=F_t\otimes_R I$ with an ideal of $A_t$ and $\Lambda_{F_t}/I_{F_t}$ with $(\Lambda/I)_{F_t}$.  Then $I_{F_t}$ is a nilpotent ideal of $A_t$ and  $I_{F_t}$ acts as $0$ on each $L_{i,F_t}$. By the case already considered, for all but finitely many values of $t$, the modules $L_{1,F_t},\ldots,L_{n,F_t}$ form a compete set of pairwise non-isomorphic simple $\Lambda_{F_t}/I_{F_t}$-modules all of which are absolutely irreducible.  Hence the modules $L_{1,F_t},\ldots,L_{n,F_t}$ form a compete set of pairwise non-isomorphic simple  $A_t$-modules and $A_t$ is split.

\bigskip

 Let $T^0$ be the set of those $t\in T$ such that $\Lambda_{F_t}$ is
split and   $L_{1,F_t},\ldots,L_{n,F_t}$ form  a complete set of
pairwise non-isomorphic
$\Lambda_{F_t}$-modules. For $t\in T^0$ and $1\leq i\leq n$,  we
set $V_{it}=L_{i,F_t}$. Thus $V_{1t},\ldots,V_{nt}$ is a complete set of pairwise non-isomorphic irreducible $A_t$-modules, for $t\in T^0$.
For $1\leq i,j\leq n$ we have the Cartan invariant   $c_{ij}$ of $A$,
i.e., the composition multiplicity of $V_j$ in the projective cover of $V_i$.  For $t\in T^0$, we have the corresponding Cartan invariant $c^t_{ij}$ of $\Lambda_{F_t}$.

\bigskip

{(2) \sl For all but finitely many values of $t\in T^0$ we have $c_{ij}=c^t_{ij}$ for all $1\leq i,j\leq n$ and in particular the modules $V_k$ and $V_l$ belong to the same block if and only if the modules $V_{kt}$ and $V_{lt}$ belong to the same block, for $1\leq k,l\leq n$.}

\bigskip

For $1 \leq i\leq n$ we let $P_i$ be the projective cover of $V_i$.
For $t\in T^0$ we denote by $P_{it}$ the projective cover of
$V_{it}$.   We choose a decomposition of $1\in A$ as an orthogonal sum
of primitive idempotents, $1=\sum_{i=1}^n\sum_{j=1}^{d_i}e_{ij}$ such that $Ae_{ij}$ is isomorphic to $P_i$, for $1\leq i\leq n$, $1\leq j\leq d_i$. We choose $0\neq h\in R$ such that $he_{ij}\in \Lambda$, for all $i,j$.

We define $T^1$ to be  the set of $t\in T^0$ such that $h\not\in M_t$. Thus the set $T^1$ is cofinite in $T$.
We have the lattice $Y_{ij}=h\Lambda e_{ij}$ in $Ae_{ij}$, for all $i,j$, and the lattice $Y=\bigoplus_{ij} Y_{ij}$ in $A$.

For $t\in T^1$, the inclusion $Y\to \Lambda$ induces an isomorphism $Y_{F_t}\to \Lambda_{F_t}$. Thus each $Y_{ij,F_t}=F_t\otimes_R Y_{ij}$ is a non-zero projective $A_{F_t}$-module. Moreover, the number of summands in a $\Lambda_{F_t}$-module decomposition of  the left regular module $\Lambda_{F_t}$ as a direct sum of indecomposable projective modules is $\sum_{i=1}^nd_i^2$. Hence each $Y_{ij,F_t}$ is indecomposable and projective.

We fix $1\leq i\leq n$ and $1\leq j\leq d_i$. Let $W$ be the maximal submodule of $Ae_{ij}$ and let $H=W\bigcap Y_{ij}$. Then $Ae_{ij}/W$ is isomorphic to $V_i$. Hence $Y_{ij}/H$ is isomorphic to a lattice in  $V_i$ and $(Y_{ij}/H)_{F_t}$ is isomorphic to $V_{it}$, by (1). Thus the projective indecomposable $A_t$-module $Y_{ij,F_t}$ has $V_{it}$ as a homomorphic image and hence $Y_{ij,F_t}$ is a projective cover of $V_{it}$, i.e. we have $Y_{ij,F_t}\cong P_{it}$.

By (1) the decomposition map $d_t:\Grot(A)\to \Grot(\Lambda_{F_t})$ is an isomorphism, taking $[V_i]$ to $[V_{it}]$,
$1\leq i\leq n$.   Moreover,   $d_t$ takes $[P_i]$ to $[Y_{ij,F_t}]=
[P_{it}]$.  Now we have $[P_i]=\sum_{j=1}^n c_{ij}[V_j]$ and applying $d_t$ we obtain $[P_{it}]=\sum_{j=1}^n c_{ij} [V_{jt}]$, which shows that $c_{ij}=c^t_{ij}$, for all $1\leq i,j\leq n$ and $t\in T^1$.

\subsection{The blocks of the Brauer algebra in characteristic $0$}\label{ss.CdVM}
In this final subsection we give a Lie theoretic proof of the block result of Cox, De Visscher and Martin. We assume that ${\rm char}\,k=0$. Furthermore $\delta$ is an arbitrary integer, $r$ is an integer $\ge0$ and we define $\hat\rho$ and the star action of the Weyl group $W(D_r)\subseteq W(C_r)$ as in Section~\ref{s.jantzen}. If $r$ is even $\ge2$ and $\delta=0$, we extend the block relation of $B_r(\delta)$ to all of $\Lambda_0^+(r,r)$ as in Subsection~\ref{ss.sameblocks}. We will apply the general results of Subsection~\ref{ss.redmodp} to the case that $R=\mb Z$, $K=\mb Q$, $A=B_r(\delta)_{\mb Q}$ and $\Lambda=B_r(\delta)_{\mb Z}$.
\begin{thm}[{\cite[Thm.~4.2]{CdVM2}}]
Let $\lambda,\mu\in\Lambda_0^+(r,r)$. Then $\lambda$ and $\mu$ are in the same $B_r(\delta)$-block if and only if $\lambda'$ and $\mu'$ are conjugate under the star action of $W(D_r)$.
\end{thm}
\begin{proof}
We have $B_r(\delta)=k\otimes_{\mb Q}B_r(\delta)_{\mb Q}$. Since $B_r(\delta)_{\mb Q}$ is split, i.e. all irreducibles are absolutely irreducible. %see Prop. 3.2 in Graham and Lehrer
From this one deduces easily that the block relation of $B_r(\delta)$ is the same as that of $B_r(\delta)_{\mb Q}$. The same holds for $B_r(\delta)_{\mb F_p}$ and $B_r(\delta)_{\ov{\mb F}_p}$, where $p$ is any prime and $\ov{\mb F}_p$ is the algebraic closure of the prime field $\mb F_p$. By (1) and (2) in Subsection~\ref{ss.redmodp} we can choose a prime $p>2(|\delta|+2r)$ such that the irreducibles of $B_r(\delta)_{\mb F_p}$ are the reductions mod $p$ of the irreducibles of $B_r(\delta)_{\mb Q}$ and such that both algebras have the same block relation. Note that ${\mc S}(\lambda)_{\mb F_p}$ is a reduction mod $p$ of ${\mc S}(\lambda)_{\mb Q}$. One easily checks that ${\mc D}(\lambda)_{\mb F_p}:={\rm hd}\,{\mc S}(\lambda)_{\mb F_p}$ is the reduction mod $p$ of ${\mc D}(\lambda)_{\mb Q}$, $\lambda\ne\emptyset$ in case $r$ is even $\ge2$ and $\delta=0$. Since, by Subsection~\ref{ss.redmodp}(1), the decomposition homomorphism $\Grot(B_r(\delta)_{\mb Q})\to \Grot(B_r(\delta)_{\mb F_p})$ is an isomorphism, we have that  ${\mc D}(\lambda)_{\mb F_p}$ is a composition factor of ${\mc S}(\mu)_{\mb F_p}$ if and only if ${\mc D}(\lambda)_{\mb Q}$ is a composition factor of ${\mc S}(\mu)_{\mb Q}$. The result now follows immediately from Corollary~\ref{cor.brauerblocks} of Theorem~\ref{thm.sameblocks}.
\end{proof}

\begin{remnn}
Since the Brauer algebra is cellular over $\mb Z$, one can actually deduce equality of the block relations of $B_r(\delta)_{\mb F_p}$ and $B_r(\delta)_{\mb Q}$ whenever the irreducibles of $B_r(\delta)_{\mb F_p}$ are the reductions mod $p$ of the irreducibles of $B_r(\delta)_{\mb Q}$.
\end{remnn}

\noindent{\it Acknowledgement}. We would like to thank A.~Cox for mentioning Proposition~\ref{prop.notinFp} to us. Furthermore, we acknowledge funding from a research grant from The Leverhulme Trust and the second author acknowledges funding from EPSRC Grant EP/C542150/1.

\bigskip

{\sc\noindent Department of Mathematics,
University of York,
Heslington, York, UK, YO10~5DD.
{\it E-mail addresses : }{\tt sd510@york.ac.uk, rht502@york.ac.uk}
}


\begin{thebibliography}{99}
%\bibitem{anjan} H.\ H.\ Andersen and J.\ C.\ Jantzen, {\it Cohomology of induced representations for algebraic groups}, Math. Ann. {\bf 269} (1984), no. 4, 487-525.
\bibitem{AdRyb} A.\ M.\ Adamovich, G.\ L.\ Rybnikov, {\it Tilting modules for classical groups and Howe duality in positive characteristic}, Transform. Groups {\bf 1} (1996), no. 1-2, 1-34.
\bibitem{ARS} M.\ Auslander, I.\ Reiten, S.\ Smal{\o}, {\it Representation theory of Artin algebras}, Cambridge Studies in Advanced Mathematics, 36, Cambridge University Press, Cambridge, 1995.
\bibitem{BenCa} D.\ J.\ Benson, J.\ F.\ Carlson, {\it Nilpotent elements in the Green ring}, J. Algebra {\bf 104} (1986), no. 2, 329-350.
%\bibitem{Ben} D.\ J.\ Benson, {\it Representations and cohomology. I. Basic representation theory of finite groups and associative algebras}, Second edition, Cambridge Studies in Advanced Mathematics, 30, Cambridge University Press, Cambridge, 1998.
\bibitem{Br} R.\ Brauer, {\it On algebras which are connected with the semisimple continuous groups}, Ann. of Math. (2) {\bf 38} (1937), no. 4, 857-872.
\bibitem{Brown1} W.\ P.\ Brown, {\it An algebra related to the orthogonal group}, Michigan Math. J. {\bf 3} (1955), 1-22.
\bibitem{Brown2} W.\ P.\ Brown, {\it The semisimplicity of $\omega\sb f\sp n$}, Ann. of Math. (2) {\bf 63} (1956), 324-335.
\bibitem{Brun} J.\ Brundan, {\it Dense orbits and double cosets}, Algebraic groups and their representations (Cambridge, 1997), 259-274.
\bibitem{Cliff} G.\ Cliff {\it A basis of bideterminants for the coordinate ring of the orthogonal group}, preprint.
\bibitem{CdVM1} A.\ Cox,  M.\ De Visscher, P.\ Martin, {\it The blocks of the Brauer algebra in characteristic zero},  preprint.
\bibitem{CdVM2} A.\ Cox,  M.\ De Visscher, P.\ Martin, {\it A geometric characterisation of the blocks of the Brauer algebra},  preprint.
%\bibitem{CdVDM} A.\ Cox,  M.\ De Visscher, S.\ Doty, P.\ Martin, {\it On the blocks of the walled Brauer algebra}, to appear in J.~Algebra.
\bibitem{DeCProc} C.\ De Concini, C.\ Procesi, {\it A characteristic free approach to invariant theory}, Advances in Math. {\bf 21} (1976), no. 3, 330-354.
\bibitem{Dotyetal} R.\ Dipper, S.\ Doty, J.\ Hu, {\it Brauer algbras, symplectic Schur algebras and Schur-Weyl duality}, Trans. Amer. Math. Soc. {\bf 360} (2008), no. 1, 189-213.
\bibitem{Don1} S.\ Donkin, {\it On Schur algebras and related algebras, I}, J. Algebra 104 (1986), no. 2, 310-328.
\bibitem{Don2} \bysame, {\it On Schur algebras and related algebras, II}, J. Algebra 111 (1987), no. 2, 354-364.
\bibitem{Don3} \bysame, {\it Good filtrations of rational modules for reductive groups} in {\it The Arcata Conference on Representations of Finite Groups} (Arcata, Calif., 1986), Proc. Sympos. Pure Math. {\bf 47:1}, Amer. Math. Soc., Providence, RI, 1987, 69-80.
\bibitem{Don4} \bysame, {\it Representations of symplectic groups and the symplectic tableaux of R. C. King}, Linear and Multilinear Algebra {\bf 29} (1991), no. 2, 113-124.
\bibitem{Don5} \bysame, {\it On tilting modules and invariants for algebraic groups}, Finite-dimensional algebras and related topics (Ottawa, ON, 1992), 59-77, NATO Adv. Sci. Inst. Ser. C Math. Phys. Sci. 424, Kluwer Acad. Publ., Dordrecht, 1994.
\bibitem{Don6} \bysame, {\it On tilting modules for algebraic groups}, Math. Z. {\bf 212} (1993), no. 1, 39-60.
\bibitem{Don7} \bysame, {\it The $q$-Schur algebra}, LMS Lecture Note Series 253, Cambridge University Press, Cambridge, 1998
\bibitem{Don8} \bysame, {\it Tilting modules for algebraic groups and finite dimensional algebras}, Handbook of Tilting Theory 215-257, LMS Lecture Note Series 332, Cambridge Univ. Press, Cambridge.
\bibitem{DorHanWal} W.\ F.\ Doran, P.\ Hanlon, D.\ Wales, {\it On the semisimplicity of the Brauer centralizer algebras}, J. Algebra {\bf 211} (1999), no. 2, 647-685.
\bibitem{Doty} S.\ Doty, {\it Polynomial representations, algebraic monoids, and Schur algebras of classical type}, J. Pure Appl. Algebra {\bf 123} (1998), no. 1-3, 165-199.
\bibitem{DotyHu} S.\ Doty and J.\ Hu, {\it Schur-Weyl duality for orthogonal groups}, Preprint.
\bibitem{ErdSa} K.\ Erdmann, C.\ S\'aenz, {\it On standardly stratified algebras}, Comm. Algebra {\bf 31} (2003), no. 7, 3429-3446.
\bibitem{Green} J.\ A.\ Green, {\it Polynomial representations of ${\rm GL}\sb{n}$}, Lecture Notes in Mathematics, {\bf 830}, Springer-Verlag, Berlin-New York, 1980.
\bibitem{HanWal} P.\ Hanlon, D.\ Wales, {\it On the decomposition of Brauer's centralizer algebras}, J. Algebra {\bf 121} (1989), no. 2, 409-445.
\bibitem{HarPag} R.\ Hartmann, R.\ Paget, {\it Young modules and filtration multiplicities for Brauer algebras}, Math. Z. {\bf 254} (2006), no. 2, 333-357.
\bibitem{James} G.\ D.\ James,  {\it The representation theory of the symmetric groups}, Lecture Notes in Mathematics, 682, Springer, Berlin, 1978.
\bibitem{Jan} J.\ C.\ Jantzen, {\it Representations of algebraic groups},
Pure and Applied Math., vol.~131. Academic Press, Boston, 1987.
\bibitem{King} R.\ C.\ King, {\it Weight multiplicities for the classical groups}, Group theoretical methods in physics (Fourth Internat. Colloq., Nijmegen, 1975), Lecture Notes in Phys., Vol. 50, Springer, Berlin, 1976, 490-499.
\bibitem{KingWelsh} R.\ C.\ King, T.\ A.\ Welsh, {\it Construction of orthogonal group modules using tableaux}, Linear and Multilinear Algebra {\bf 33} (1993), no. 3-4, 251-283.
\bibitem{koxi} S.\ K\"onig, C.\ Xi, {\it A characteristic free approach to Brauer algebras}, Trans. Amer. Math. Soc. {\bf 353} (2001), no. 4, 1489-1505.
%\bibitem{Lit} D.\ E.\ Littlewood, {\it The Theory of Group Characters and Matrix Representations of Groups}, Oxford University Press, New York, 1940.
\bibitem{Mac} I.\ G.\ Macdonald, {\it Symmetric functions and Hall polynomials}, Second edition, Oxford Mathematical Monographs, Oxford Science Publications, Oxford University Press, New York, 1995.
\bibitem{MaWo} P.\ Martin, D.\ Woodcock, {\it The partition algebras and a new deformation of the Schur algebras}, J. Algebra {\bf 203} (1998), no. 1, 91-124.
\bibitem{smartin} S.\ Martin, {\it Schur algebras and representation theory}, Cambridge Tracts in Mathematics, 112, Cambridge University Press, Cambridge, 1993.
\bibitem{Oe} S.\ Oehms, {\it Centralizer coalgebras, FRT-construction, and symplectic monoids}, J. Algebra {\bf 244} (2001), no. 1, 19-44.
\bibitem{Rot} J.\ J.\ Rotman, {\it An introduction to homological algebra}, Pure and Applied Mathematics, 85, Academic Press, New York-London, 1979.
\bibitem{Serre1} J-P. Serre, {\it Repr\'esentations lin\'eaires des groupes finis}, Hermann, Paris 1967.
\bibitem{Serre2} J-P. Serre, {\it Semisimplicity and tensor products of group representations: converse theorems}, with an appendix by Walter Feit, J. Algebra \textbf{194} (1997), no. 2, 496-520.
\bibitem{T} R.\ H.\ Tange, {\it The symplectic ideal and a double centraliser theorem}, J.~London Math. Soc. {\bf 77} (2008), no. 3, 687-699.
\bibitem{Wen} H.\ Wenzl, {\it On the structure of Brauer's centralizer algebras}, Ann. of Math. (2) {\bf 128} (1988), no. 1, 173-193.
\bibitem{Weyl} H.\ Weyl, {\it The classical groups, Their invariants and representations}, second edition, Princeton University Press, 1946.
%\bibitem{Wong} W.\ J.\ Wong, {\it Representations of Chevalley groups in characteristic $p$}, Nagoya Math. J. {\bf 45} (1972), 39-78.
\end{thebibliography}
\end{document}